\documentclass[letterpaper,11pt]{article}
\usepackage[utf8]{inputenc}
\usepackage{amsmath}
\usepackage{fancyhdr}
\usepackage{amsthm}
\usepackage{amsfonts}
\usepackage{amssymb}
\usepackage[margin=1in]{geometry}
\usepackage{tikz}
\usepackage{cite}
\usepackage{comment}
\usepackage{float}
\usepackage[colorlinks=true]{hyperref}
\usepackage{enumitem}
\setlist{nolistsep}

\tikzset{
  pt/.style={insert path={node[scale=0.5]{.}}},
  ptt/.style={insert path={node[scale=3]{.}}},
  du/.style={insert path={ [pt] .. controls +(0,1) and +(0,-1) .. +(#1,2) [pt]}},
  dd/.style={insert path={ [pt] .. controls +(0,1) and +(0,1) .. +(#1,0) [pt]}},
  uu/.style={insert path={ [pt] .. controls +(0,-1) and +(0,-1) .. +(#1,0) [pt]}},
  dm/.style={insert path={ [pt] -- +(#1,-1) [ptt]}},
  um/.style={insert path={ [pt] -- +(#1,1) [ptt]}},
  dq/.style={insert path={ [pt] -- +(#1,-0.75) [ptt]}},
  uq/.style={insert path={ [pt] -- +(#1,0.75) [ptt]}},
  dum/.style={insert path={ [pt] .. controls +(0,0.5) and +(0,-0.5) .. +(#1,1) [pt]}},
  udm/.style={insert path={ [pt] .. controls +(0,-0.5) and +(0,0.5) .. +(#1,-1) [pt]}},
  ddm/.style={insert path={ [pt] .. controls +(0,0.4) and +(0,0.4) .. +(#1,0) [pt]}},
  uum/.style={insert path={ [pt] .. controls +(0,-0.4) and +(0,-0.4) .. +(#1,0) [pt]}},
}

\newcommand{\tikzinline}[1]{\vcenter{\hbox{#1}}}

\newtheorem{theoremsimplenumber}{Theorem}
\newtheorem*{theorem*}{Theorem}
\newtheorem{theorem}{Theorem}[section]

\newtheorem{lemma}[theorem]{Lemma}
\newtheorem{proposition}[theorem]{Proposition}
\newtheorem{conjecture}[theorem]{Conjecture}

\theoremstyle{definition}

\newtheorem{definition}[theorem]{Definition}
\theoremstyle{remark}
\newtheorem{remark}[theorem]{Remark}

\title{\texorpdfstring{On the Structure and Generators \\ of the $n$th-order Chromatic Algebra}{On the Structure and Generators \\ of the nth-order Chromatic Algebra}}
\author{Ethan Liu}

\begin{document}

\maketitle
\begin{abstract}

This work investigates the intrinsic properties of the chromatic algebra, introduced by Fendley and Krushkal as a framework to study the chromatic polynomial. We prove that the dimension of the $n$th-order chromatic algebra is the $2n$th Riordan number, which exhibits exponential growth. We find a generating set of size $\binom{n}{2}$, and we provide a procedure to construct the basis from the generating set. We additionally provide proofs for fundamental facts about this algebra that appear to be missing from the literature. These include determining a representation of the chromatic algebra as noncrossing planar partitions and expanding the chromatic relations to include an edge case.
\end{abstract}

\section{Introduction}

A graph is a collection of vertices and edges, where edges connect two (not necessarily distinct) vertices. One important property of a graph is its graph colorings, which are ways to assign at most a fixed number of colors to the graph's vertices such that no edge can connect two vertices with the same color. The number of graph colorings of a graph $G$, as a function of the number of colors permitted, coincides with a polynomial function called the \emph{chromatic polynomial} $\chi_{G}(Q)$. We aim to prove fundamental results about the chromatic algebra, which provides an algebraic framework to study the chromatic polynomial.

The chromatic algebra is an associative algebra which combines the concept of graph colorings with abstract algebra. This algebra was introduced by Fendley and Krushkal after they noticed that the chromatic polynomial naturally arises in the low-temperature expansion of the Potts model \cite{fendley2010link}. It is closely related to the Temperley--Lieb algebra, which is important in connecting quantum field theory, knot theory, and computational logic \cite{abramsky2007temperley}. It also relates to the SO(3) Birman--Murakami--Wenzl 
algebras \cite{fendley2010link,birman1989braids,murakami1987kauffman}.

The chromatic algebra has already proven useful in studying the chromatic polynomial. Previously, W.~T.~Tutte discovered several ``golden identities" of the chromatic polynomial when evaluating at the special value $Q=\phi+1$, where $\phi = \tfrac{1+\sqrt{5}}{2}$ denotes the golden ratio. One such identity is formulated in \cite[(11.15)]{tutte1998graph} as follows: for related planar graphs $Y_1$, $Y_2$, $Z_1$, and $Z_2$ with the relationship given by \cite[Figure~11.1]{tutte1998graph},
\begin{align*}
    \chi_{Z_1}(\phi+1)+\chi_{Z_2}(\phi+1)=\phi^{-3}\left(\chi_{Y_1}(\phi+1)+\chi_{Y_2}(\phi+1)\right).
\end{align*}
Fendley and Krushkal used the chromatic algebra to demonstrate that this identity arises from a Jones--Wenzl projector in the Temperley--Lieb algebra \cite[Section~2]{fendley2009tutte}, and generalized it to a larger set of values $Q = 2 + 2 \cos{(\tfrac{2\pi j}{n+1})}$ for positive integers $j < n$ \cite[Lemma~5.3]{fendley2009tutte}.

Furthermore, the chromatic polynomial as it arises in the Potts model has significance in statistical mechanics. Specifically, a generalization of the chromatic polynomial in two indeterminates, known as the \emph{dichromatic polynomial} or \emph{Tutte polynomial}, is the partition function of the Potts model \cite[Section~12.2]{baxter2016exactly}, and its zeroes represent possible points of physical phase transitions \cite[Section~4]{yang1952statistical}. In the Potts model's low-temperature case, this partition function specializes to the chromatic polynomial. Here, the chromatic algebra describes the model's geometric degrees of freedom \cite[Section~3]{fendley2010link}. Thus the chromatic algebra is deeply linked to many branches of mathematics.

In this investigation, we prove fundamental results about the structure of the chromatic algebra, including its dimension, and we provide a generating set for the chromatic algebra viewed as a finitely-generated (even finite-dimensional) algebra.

In Section~\ref{sec:background}, we provide background definitions leading up to the chromatic algebra. We then introduce, in Section~\ref{sec:chromatic-algebra}, a rigorous definition of the chromatic algebra, resolving some ambiguities and typos in the exposition of Fendley and Krushkal \cite{fendley2010link}. Section~\ref{sec:dimension} and Section~\ref{sec:generating-set} introduce our first and second main results determining the dimension and a multiplicative generating set, respectively, of the chromatic algebra.

Our main results are stated below.
\begin{theoremsimplenumber}
    \label{thm:dimension}
    The dimension of $\mathcal{C}_n$ is the $2n$th Riordan number $R_{2n}$.
\end{theoremsimplenumber}
\begin{theoremsimplenumber}
    \label{thm:chromatic-generating-set}
    Let $i,j,n$ be positive integers such that $1 \leq i < j \leq n$. Let $e_{i,j}^{n} \in B_n$ denote the diagram in which the all top and bottom boundary points between the $i$th from the left and the $j$th from the left, inclusive, are all connected to one inner vertex, and all other boundary points are paired with the corresponding boundary point on the opposite side by vertical edges. Let $E_n$ be the subset of $B_n$ containing all possible $e_{i,j}^{n}$. The set $E_n$ generates $\mathcal{C}_n$.
\end{theoremsimplenumber}

\section{Background}
\label{sec:background}

In this section, we introduce background definitions required to understand the chromatic algebra. These notions relate to associative algebras and graphs, especially planar graphs. This section culminates with the definition of chromatic diagrams and the free algebra $\mathcal{F}_n$.

\subsection{Associative Algebras}

    An \emph{associative algebra} $A$ over a field $K$, or simply a \emph{$K$-algebra}, is a set of elements which is closed under compatible operations of addition ($+$), multiplication ($\times$), and multiplication by scalars from $K$ ($\cdot$). These operations satisfy the following properties, given elements $a,b\in A$, $k \in K$:
    \begin{itemize}
        \item $(A,+,\times)$ is a ring.
        \item $(A,K,+,\cdot)$ is a vector space.
        \item $k \cdot (a \times b) = (k \cdot a) \times b = a \times (k \cdot b)$.
    \end{itemize}
    Essentially, an associative algebra is a ring whose underlying linear structure is a vector space.
    A \emph{basis} of a $K$-algebra is a minimal subset such that each element of the algebra can be uniquely written as a $K$-linear combination of the elements in the basis.
    The \emph{dimension} of a $K$-algebra is the size of any given basis. 
A \emph{generating set} for a $K$-algebra $A$ is a subset $S\subset A$ for which every element of $A$ can be written as a polynomial with $K$-coefficients in the variables $S$. In other words, $S$ should generate $A$ by scalar multiplication, multiplication, and addition.

\subsection{Graphs}

    A \emph{graph} is a (possibly empty) collection of points and lines connecting pairs of (not necessarily distinct) points. The points and lines are commonly known as \emph{vertices} and \emph{edges}, respectively.

\begin{figure}[H]
    \begin{align*}
        \tikzinline{
        \begin{tikzpicture}
            \node[draw,circle,fill,scale=0.2] (A) at (0,0) {$ $};
            \node[draw,circle,fill,scale=0.2] (B) at (1,0) {$ $};
            \node[draw,circle,fill,scale=0.2] (C) at (1,1) {$ $};
            \node[draw,circle,fill,scale=0.2] (D) at (0,1) {$ $};
            \draw (A) -- (B) -- (C) -- (D) -- (A);
            \draw (A) -- (C);
            \draw (B) -- (D);
        \end{tikzpicture}
        }
        \hspace{20px}
        \tikzinline{
        \begin{tikzpicture}
            \node[draw,circle,fill,scale=0.2] (A) at (0,0) {$ $};
            \node[draw,circle,fill,scale=0.2] (B) at (1,0) {$ $};
            \node[draw,circle,fill,scale=0.2] (C) at (2,0) {$ $};
            \node[draw,circle,fill,scale=0.2] (D) at (3,0) {$ $};
            \draw (A) -- (B) -- (C) -- (D);
            \draw (C) .. controls ++(0.5,1) and ++(-0.5,1) .. (C);
        \end{tikzpicture}
        }
        \hspace{20px}
        \tikzinline{
        \begin{tikzpicture}
            \node[draw,circle,fill,scale=0.2] (A) at (0,5) {$ $};
            \node[draw,circle,fill,scale=0.2] (B) at (1,4) {$ $};
            \node[draw,circle,fill,scale=0.2] (C) at (-1,4) {$ $};
            \node[draw,circle,fill,scale=0.2] (D) at (1.5,3) {$ $};
            \node[draw,circle,fill,scale=0.2] (E) at (-0.5,3) {$ $};
            \node[draw,circle,fill,scale=0.2] (F) at (-1.5,3) {$ $};
            \node[draw,circle,fill,scale=0.2] (G) at (0.5,3) {$ $};
            \draw (A) -- (B) -- (D);
            \draw (B) -- (G);
            \draw (A) -- (C) -- (E);
            \draw (C) -- (F);
        \end{tikzpicture}
        }
    \end{align*}
    \caption{Examples of graphs.}
    \label{fig:graph}
\end{figure}
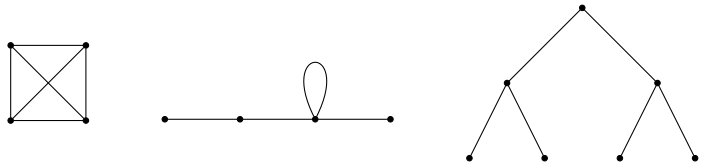


    The \emph{valency} of a vertex is the number of times that edges connect to that vertex. We say a vertex is \emph{$k$-valent} if its valency is $k$.
    Our definition of a graph permits self-looping edges, which connect to the same vertex twice. We consider the valency of the vertex to be increased by two in the event of a self-looping edge. We also allow multi-edges, which increase the valency of the two vertices by as many edges as there are between them. A \emph{connected component} of a graph is a subset of vertices and edges such that from each vertex or edge, there exists a path to every vertex or edge in the component but no path to any vertices or edges outside it. A \emph{plane graph}, or a \emph{planar embedding of a graph} is a drawing of a graph in 2-dimensional space such that no two edges intersect. Plane graphs are unique up to \emph{isotopy} (continuous deformation of the graph without crossing edges), denoted by $\sim$. Isotopy of plane graphs is demonstrated in Figure~\ref{fig:isotopy}. In addition, the edges of plane graphs separate the 2-dimensional space into disjoint \emph{faces}.

\begin{figure}[H]
    \begin{align*}
        \tikzinline{
        \begin{tikzpicture}
            \node[draw,circle,fill,scale=0.2] (A) at (-1,0) {$ $};
            \node[draw,circle,fill,scale=0.2] (B) at (0,0) {$ $};
            \node[draw,circle,fill,scale=0.2] (C) at (1,1) {$ $};
            \node[draw,circle,fill,scale=0.2] (D) at (1,-1) {$ $};
            \draw (A) -- (B) -- (C) -- (D) -- (B);
        \end{tikzpicture}
        }
        \hspace{2ex}
        \sim
        \hspace{2ex}
        \tikzinline{
        \begin{tikzpicture}
            \node[draw,circle,fill,scale=0.2] (A) at (0,1) {$ $};
            \node[draw,circle,fill,scale=0.2] (B) at (0,0) {$ $};
            \node[draw,circle,fill,scale=0.2] (C) at (1,1) {$ $};
            \node[draw,circle,fill,scale=0.2] (D) at (1,-1) {$ $};
            \draw (A) -- (B) .. controls ++(0,0.2) and ++(-0.2,0) .. (C) .. controls ++(0.2,-0.4) and ++(0.2,0.4) .. (D) .. controls ++(-0.2,0) and ++(0,-0.2) .. (B);
        \end{tikzpicture}
        }
        \hspace{2ex}
        \not\sim
        \hspace{2ex}
        \tikzinline{
        \begin{tikzpicture}
            \node[draw,circle,fill,scale=0.2] (A) at (0.4,0) {$ $};
            \node[draw,circle,fill,scale=0.2] (B) at (0,0) {$ $};
            \node[draw,circle,fill,scale=0.2] (C) at (1,1) {$ $};
            \node[draw,circle,fill,scale=0.2] (D) at (1,-1) {$ $};
            \draw (A) -- (B) .. controls ++(0,0.2) and ++(-0.2,0) .. (C) .. controls ++(0.2,-0.4) and ++(0.2,0.4) .. (D) .. controls ++(-0.2,0) and ++(0,-0.2) .. (B);
        \end{tikzpicture}
        }
    \end{align*}
\caption{Isotopy and non-isotopy between three plane graphs.}
\label{fig:isotopy}
\end{figure}
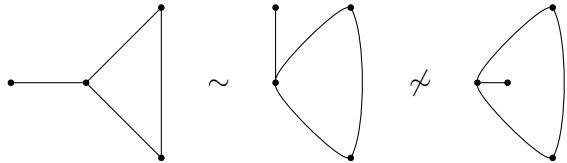

\begin{remark}
    Not all graphs can be embedded on a plane. In this paper, we only consider plane graphs.
\end{remark}

Let $G$ be a plane graph. Then, a graph $\widehat{G}$ is \emph{dual} to $G$ if it has a vertex for each face of $G$ and an edge between each adjacent pair of faces in $G$. We may also say that $\widehat{G}$ is a \emph{dual graph} of $G$. Note that $\widehat{G}$ is again a planar graph, and unique up to isotopy.





\subsection{Chromatic Diagram}

\begin{definition}
    An \emph{$n$th-order chromatic diagram} is a planar embedding of a graph, up to isotopy, within a rectangle such that there are $n$ marked points on the top edge of the rectangle and $n$ marked points on the bottom edge of the rectangle, and the graph is embedded in the interior of the rectangle and possibly connected to these marked \emph{boundary points}. Each boundary point must be $1$-valent.
\end{definition}
\begin{figure}[H]
    \begin{center}
        \begin{tikzpicture}[scale=0.9]
            \draw[dashed] (0,0) rectangle (6,2);
            \node[draw,circle,fill,scale=0.2] (A) at (1,1) {$ $};
            \node[draw,circle,fill,scale=0.2] (B) at (3,1) {$ $};
            \node[draw,circle,fill,scale=0.2] (C) at (4,0.7) {$ $};
            \node[draw,circle,fill,scale=0.2] (D) at (4,1.3) {$ $};
            
            \node[draw,circle,scale=0.4] (B1) at (1,0) {$ $};
            \node[draw,circle,scale=0.4] (B2) at (2,0) {$ $};
            \node[draw,circle,scale=0.4] (B3) at (3,0) {$ $};
            \node[draw,circle,scale=0.4] (B4) at (4,0) {$ $};
            \node[draw,circle,scale=0.4] (B5) at (5,0) {$ $};
            \node[draw,circle,scale=0.4] (T1) at (1,2) {$ $};
            \node[draw,circle,scale=0.4] (T2) at (2,2) {$ $};
            \node[draw,circle,scale=0.4] (T3) at (3,2) {$ $};
            \node[draw,circle,scale=0.4] (T4) at (4,2) {$ $};
            \node[draw,circle,scale=0.4] (T5) at (5,2) {$ $};
            
            \draw (1,0) -- (A) -- (1,2);
            \draw (2,0) .. controls ++(0,0.5) and ++(0.25,-0.25) .. (A);
            \draw (A) -- (B);
            \draw (3,0) -- (B);
            \draw (3,2) .. controls ++(0,-0.5) and ++(-0.25, 0.25) .. (D);
            \draw (2,2) .. controls ++(0,-0.5) and ++(0,0.25) .. (B);
            \draw (C) -- (B) -- (D);
            \draw (4,0) -- (4,2);
            \draw (5,0) -- (5,2);
        \end{tikzpicture}
    \end{center}
    \caption{A fifth-order chromatic diagram, with boundary points circled.}
    \label{fig:chromatic-diagram}
\end{figure}
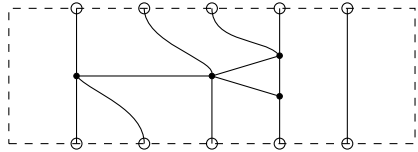

\begin{remark}
    Just as plane graphs are unique up to isotopy, chromatic diagrams are also unique up to isotopy.
\end{remark}

\subsection{\texorpdfstring{The Free Algebra $\mathcal{F}_n$}{The Free Algebra Fn}}

\begin{definition}
    Let $\mathcal{G}_n$ be the set of all $n$th-order chromatic diagrams.
\end{definition}

Henceforth we use the field $\mathbb{C}((Q))$ of complex Laurent series in the indeterminate $Q$, essentially power series which can have finitely many negative power terms.
\begin{definition}
    Let $\mathcal{F}_n$ denote the associative algebra over $\mathbb{C}((Q))$ with free additive generators given by the elements of $\mathcal{G}_n$, where multiplication is given by vertical stacking.
\end{definition}

More precisely, since planar embeddings are defined up to isotopy, multiplication can be defined by choosing planar embeddings of the two diagrams, then stacking, and then taking the isotopy class of the resulting diagram. Specifically, two chromatic diagrams $A$ and $B$ are multiplied by stacking $A$ on top of $B$ (after choosing planar embeddings for which the boundary points coincide), connecting the edges which meet at the bottom boundary points of $A$ and the top boundary points of $B$, deleting the coinciding boundary points, and removing the rectangular boundary formed by the bottom of $A$ and the top of $B$. This multiplication is demonstrated in Figure~\ref{fig:diagram-multiplication}.
\begin{figure}[H]
    \begin{align*}
        \tikzinline{
        \begin{tikzpicture}[scale=0.5]
            \draw[dashed] (0,0) rectangle (4,2);
            \draw (1,0) -- (1,2);
            \draw (2,0) .. controls ++(0,1) and ++(0,1) .. (3,0);
            \draw (2,2) .. controls ++(0,-0.5) and ++(0.5,0) .. (1,1);
            \draw (3,2) .. controls ++(0,-0.5) and ++(0.5,0) .. (1,1);
            \node[draw,circle,fill,scale=0.2] (A) at (1,1) {$ $};
        \end{tikzpicture}
        } \times \tikzinline{
        \begin{tikzpicture}[scale=0.5]
            \draw[dashed] (0,0) rectangle (4,2);
            \draw (1,0) -- (1,2);
            \draw (2,0) -- (2,2);
            \draw (3,0) -- (3,2);
        \end{tikzpicture}
        }
        &= \tikzinline{
        \begin{tikzpicture}[scale=0.5]
            \draw[dashed] (0,0) rectangle (4,2);
            \draw[dashed] (0,2) rectangle (4,4);
            \draw (1,0) -- (1,4);
            \draw (2,2) .. controls ++(0,1) and ++(0,1) .. (3,2);
            \draw (2,4) .. controls ++(0,-0.5) and ++(0.5,0) .. (1,3);
            \draw (3,4) .. controls ++(0,-0.5) and ++(0.5,0) .. (1,3);
            \node[draw,circle,fill,scale=0.2] (A) at (1,3) {$ $};
            \draw (2,0) -- (2,2);
            \draw (3,0) -- (3,2);
        \end{tikzpicture}
        } = \tikzinline{
        \begin{tikzpicture}[scale=0.5]
            \draw[dashed] (0,0) rectangle (4,2);
            \draw (1,0) -- (1,2);
            \draw (2,0) .. controls ++(0,1) and ++(0,1) .. (3,0);
            \draw (2,2) .. controls ++(0,-0.5) and ++(0.5,0) .. (1,1);
            \draw (3,2) .. controls ++(0,-0.5) and ++(0.5,0) .. (1,1);
            \node[draw,circle,fill,scale=0.2] (A) at (1,1) {$ $};
        \end{tikzpicture}
        }
        \\ \\
        \tikzinline{
        \begin{tikzpicture}[scale=0.5]
            \draw[dashed] (0,0) rectangle (4,2);
            \draw (1,0) -- (1,2);
            \draw (2,0) .. controls ++(0,0.5) and ++(-0.5,0) .. (3,1);
            \draw (2,2) .. controls ++(0,-0.5) and ++(0.5,0) .. (1,1);
            \draw (3,0) -- (3,2);
            \node[draw,circle,fill,scale=0.2] (A) at (1,1) {$ $};
            \node[draw,circle,fill,scale=0.2] (B) at (3,1) {$ $};
        \end{tikzpicture}
        } \times \tikzinline{
        \begin{tikzpicture}[scale=0.5]
            \draw[dashed] (0,0) rectangle (4,2);
            \draw (1,0) .. controls ++(0,1) and ++(0,-1) .. (2,2);
            \draw (2,0) .. controls ++(0,1) and ++(0,-1) .. (1,2);
            \draw (3,0) -- (3,2);
            \node[draw,circle,fill,scale=0.2] (A) at (1.5,1) {$ $};
        \end{tikzpicture}
        }
        &=
        \tikzinline{
        \begin{tikzpicture}[scale=0.5]
            \draw[dashed] (0,0) rectangle (4,2);
            \draw[dashed] (0,2) rectangle (4,4);
            \draw (1,2) -- (1,4);
            \draw (1,0) .. controls ++(0,1) and ++(0,-1) .. (2,2);
            \draw (2,0) .. controls ++(0,1) and ++(0,-1) .. (1,2);
            \draw (3,0) -- (3,4);
            \node[draw,circle,fill,scale=0.2] (A) at (1.5,1) {$ $};
            \draw (2,2) .. controls ++(0,0.5) and ++(-0.5,0) .. (3,3);
            \draw (2,4) .. controls ++(0,-0.5) and ++(0.5,0) .. (1,3);
            \node[draw,circle,fill,scale=0.2] (C) at (1,3) {$ $};
            \node[draw,circle,fill,scale=0.2] (D) at (3,3) {$ $};
        \end{tikzpicture}
        } = \tikzinline{
        \begin{tikzpicture}[scale=0.5]
            \draw[dashed] (0,0) rectangle (4,2);
            \draw (1,1) -- (1,2);
            \draw (1,0) .. controls ++(0,0.5) and ++(0,-0.5) .. (2,1);
            \draw (2,0) .. controls ++(0,0.5) and ++(0,-0.5) .. (1,1);
            \draw (3,0) -- (3,2);
            \node[draw,circle,fill,scale=0.2] (A) at (1.5,0.5) {$ $};
            \draw (2,1) .. controls ++(0,0.25) and ++(-0.5,0) .. (3,1.5);
            \draw (2,2) .. controls ++(0,-0.25) and ++(0.5,0) .. (1,1.5);
            \node[draw,circle,fill,scale=0.2] (C) at (1,1.5) {$ $};
            \node[draw,circle,fill,scale=0.2] (D) at (3,1.5) {$ $};
        \end{tikzpicture}
        }
    \end{align*}
\caption{Multiplication of third-order chromatic diagrams in $\mathcal{F}_3$.}
\label{fig:diagram-multiplication}
\end{figure}
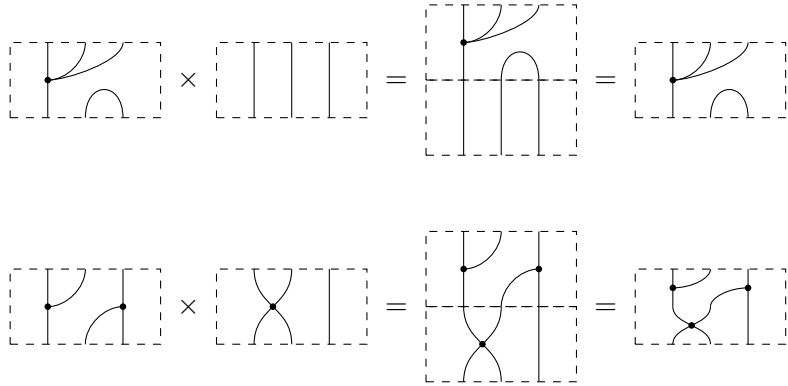

\begin{remark}
    In other words, each element of $\mathcal{F}_n$ is a formal linear combination of elements of $\mathcal{G}_n$, where all scalars come from $\mathbb{C}((Q))$. This is shown in Figure~\ref{fig:free-algebra-element}.
\end{remark}

\begin{figure}[H]
    \begin{align*}
        \frac{1}{Q-1}\tikzinline{
        \begin{tikzpicture}[scale=0.5]
            \node[draw,circle,fill,scale=0.2] (A) at (1.5,0.7) {$ $};
            \node[draw,circle,fill,scale=0.2] (B) at (1.5,1.3) {$ $};
            \draw[dashed] (0,0) rectangle (3,2);
            \draw (1,0) .. controls ++(0,0.5) and ++(-0.25,-0.25) .. (A);
            \draw (2,0) .. controls ++(0,0.5) and ++(0.25,-0.25) .. (A);
            \draw (1,2) .. controls ++(0,-0.5) and ++(-0.25,+0.25) .. (B);
            \draw (2,2) .. controls ++(0,-0.5) and ++(+0.25,+0.25) .. (B);
            \draw (A) -- (B);
        \end{tikzpicture}
        }+(1-Q)\tikzinline{
        \begin{tikzpicture}[scale=0.5]
            \draw[dashed] (0,0) rectangle (3,2);
            \draw (1,0) -- (1,2);
            \draw (2,0) -- (2,2);
        \end{tikzpicture}}
        \in \mathcal{F}_2.
        \end{align*}
\caption{An element of the free algebra $\mathcal{F}_2$.}
\label{fig:free-algebra-element}
\end{figure}

Note that $\mathcal{F}_n$ is infinite-dimensional for any $n$, because the number of $n$th-order chromatic diagrams is clearly infinite (for example, adding a zero-valent vertex to any chromatic diagram yields a distinct chromatic diagram).

\section{Chromatic Algebra}
\label{sec:chromatic-algebra}

In this section, we introduce the chromatic algebra. We also define relevant concepts such as the chromatic basis and rigorously define the chromatic algebra by resolving ambiguities in previous definitions \cite[Definition~3.1]{fendley2010link}. In particular, previous definitions did not account for certain edge cases.

\subsection{Definition of the Chromatic Algebra}

Let $G$ be a graph, and let $e$ be an edge in $G$. Then, \emph{contraction} of $e$ is its removal by merging its two endpoints into a single vertex, denoted by $G / e$. \emph{Deletion} of $e$ is its removal without otherwise modifying the graph, denoted by $G \setminus e$. Thse notions of contraction and deletion can be extended to plane graphs in chromatic diagrams. Henceforth, any mention of a graph refers to the plane graph embedded within a chromatic diagram, unless stated otherwise.
\begin{remark}
    In the context of chromatic diagrams and the free algebra, deletion is not defined for edges with an endpoint being a marked boundary point, because the resulting graph will not be a chromatic diagram. Similarly, contraction may also cause the result to not be a chromatic diagram, so we can only speak of deletion and contraction for what we call \textit{inner edges}, as defined below.
\end{remark}

    An \emph{inner vertex} in a chromatic diagram is a vertex which is not a boundary point. An \emph{inner edge} is an edge whose endpoints are both inner vertices. An \emph{outer edge} is an edge with at least one endpoint being a boundary point. A \emph{loop} is an edge whose endpoints are the same vertex.

\begin{figure}[H]
    \begin{center}
        \begin{tikzpicture}[scale=0.9]
            \draw[dashed] (0,0) rectangle (6,2);
            \node[draw,circle,fill,scale=0.2,blue] (A) at (1,1) {$ $};
            \node[draw,circle,fill,scale=0.2,blue] (B) at (3,1) {$ $};
            \node[draw,circle,fill,scale=0.2,blue] (C) at (4,0.7) {$ $};
            \node[draw,circle,fill,scale=0.2,blue] (D) at (4,1.3) {$ $};
            
            \node[draw,circle,scale=0.4,red] (B1) at (1,0) {$ $};
            \node[draw,circle,scale=0.4,red] (B2) at (2,0) {$ $};
            \node[draw,circle,scale=0.4,red] (B3) at (3,0) {$ $};
            \node[draw,circle,scale=0.4,red] (B4) at (4,0) {$ $};
            \node[draw,circle,scale=0.4,red] (B5) at (5,0) {$ $};
            \node[draw,circle,scale=0.4,red] (T1) at (1,2) {$ $};
            \node[draw,circle,scale=0.4,red] (T2) at (2,2) {$ $};
            \node[draw,circle,scale=0.4,red] (T3) at (3,2) {$ $};
            \node[draw,circle,scale=0.4,red] (T4) at (4,2) {$ $};
            \node[draw,circle,scale=0.4,red] (T5) at (5,2) {$ $};
            
            \draw[red] (1,0) -- (A) -- (1,2);
            \draw[red] (2,0) .. controls ++(0,0.5) and ++(0.25,-0.25) .. (A);
            \draw[blue] (A) -- (B);
            \draw[red] (3,0) -- (B);
            \draw[red] (3,2) .. controls ++(0,-0.5) and ++(-0.25, 0.25) .. (D);
            \draw[red] (2,2) .. controls ++(0,-0.5) and ++(0,0.25) .. (B);
            \draw[blue] (C) -- (B) -- (D);
            \draw[red] (4,0) -- (4,0.7);
            \draw[blue] (4,0.7) -- (4,1.3);
            \draw[red] (4,1.3) -- (4,2);
            \draw[red] (5,0) -- (5,2);
        \end{tikzpicture}
    \end{center}
    \caption{A fifth-order chromatic diagram. Inner edges and inner vertices are marked blue, while outer edges and boundary points are marked red.}
    \label{fig:inner-edges-outer-edges}
\end{figure}
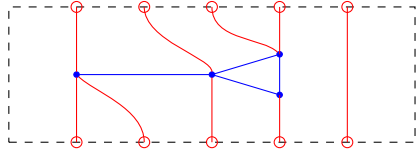


\begin{definition}
    Let $I_n$ denote the ideal of $\mathcal{F}_n$ generated by the following three equivalence relations, known as \emph{chromatic relations}:
    \begin{enumerate}[label={(\arabic*)}]
        \item If $e$ is an inner edge of a graph $G$ which is not a loop, then $G = G/e - G\setminus e$. 
        \item If $G$ contains an inner edge $e$ which is a loop, then $G = (Q-1) G \setminus e.$
        \item A 2-valent vertex $v$ may be deleted, merging its two adjacent edges.
    \end{enumerate}
\end{definition}

    These chromatic equivalence relations may look slightly unnatural, but they are defined such that when they are applied on chromatic diagrams, they preserve the chromatic polynomials of the dual graphs of these diagrams (the dual graph of a chromatic diagram is obtained by removing the rectangle and connecting the corresponding top and bottom boundary points). Thus the chromatic algebra can be thought of as a reduction from the very large free algebra which preserves the classes distinguished by the trace map, where the trace map is a function which is closely related to the chromatic polynomial of the dual graph (see \cite[Definition~3.2]{fendley2010link}).

\begin{definition}
    The \emph{$n$th-order chromatic algebra} $\mathcal{C}_n$ is the $\mathbb{C}((Q))$-algebra obtained by taking the quotient of $\mathcal{F}_n$ by $I_n$.
\end{definition}


To illustrate the chromatic relations and their importance in reducing the chromatic algebra, in Figure~\ref{fig:chromatic-relations}, we consider a chromatic diagram upon which we apply the relations sequentially to reduce it into a linear combination of more elementary chromatic diagrams.

\begin{figure}[H]
    \begin{align*}
        \tikzinline{
        \begin{tikzpicture}[scale=0.6]
            \draw[dashed] (0,0) rectangle (3,2);
            \draw (1,0) -- (1,2);
            \draw (1,1) .. controls ++(-1,1) and ++(-1,-1) .. (1,1);
            \draw (2,0) -- (2,2);
            \draw (1,1) -- (2,1);
            \node[draw,circle,fill,scale=0.2] (A) at (1,1) {$ $};
            \node[draw,circle,fill,scale=0.2] (B) at (2,1) {$ $};
            \node[draw,circle,fill,scale=0.2] (C) at (2,0.5) {$ $};
        \end{tikzpicture}
        } &= \tikzinline{
        \begin{tikzpicture}[scale=0.6]
            \draw[dashed] (0,0) rectangle (3,2);
            \draw (1.5,1) .. controls ++(-1,1) and ++(-1,-1) .. (1.5,1);
            \draw (1,0) .. controls ++(0,1) and ++(0,-1) .. (2,2);
            \draw (2,0) .. controls ++(0,1) and ++(0,-1) .. (1,2);
            \node[draw,circle,fill,scale=0.2] (A) at (1.5,1) {$ $};
            \node[draw,circle,fill,scale=0.2] (B) at (1.89,0.5) {$ $};
        \end{tikzpicture}
        } - \tikzinline{
        \begin{tikzpicture}[scale=0.6]
            \draw[dashed] (0,0) rectangle (3,2);
            \draw (1,0) -- (1,2);
            \draw (1,1) .. controls ++(-1,1) and ++(-1,-1) .. (1,1);
            \draw (2,0) -- (2,2);
            \node[draw,circle,fill,scale=0.2] (A) at (1,1) {$ $};
            \node[draw,circle,fill,scale=0.2] (B) at (2,1) {$ $};
            \node[draw,circle,fill,scale=0.2] (C) at (2,0.5) {$ $};
        \end{tikzpicture}
        }
        \\&= (Q-1)\tikzinline{
        \begin{tikzpicture}[scale=0.6]
            \draw[dashed] (0,0) rectangle (3,2);
            \draw (1,0) .. controls ++(0,1) and ++(0,-1) .. (2,2);
            \draw (2,0) .. controls ++(0,1) and ++(0,-1) .. (1,2);
            \node[draw,circle,fill,scale=0.2] (A) at (1.5,1) {$ $};
            \node[draw,circle,fill,scale=0.2] (B) at (1.89,0.5) {$ $};
        \end{tikzpicture}
        } - (Q-1)\tikzinline{
        \begin{tikzpicture}[scale=0.6]
            \draw[dashed] (0,0) rectangle (3,2);
            \draw (1,0) -- (1,2);
            \draw (2,0) -- (2,2);
            \node[draw,circle,fill,scale=0.2] (A) at (1,1) {$ $};
            \node[draw,circle,fill,scale=0.2] (B) at (2,1) {$ $};
            \node[draw,circle,fill,scale=0.2] (C) at (2,0.5) {$ $};
        \end{tikzpicture}
        }
        \\& \hspace{40px} = (Q-1)\tikzinline{
        \begin{tikzpicture}[scale=0.6]
            \draw[dashed] (0,0) rectangle (3,2);
            \draw (1,0) .. controls ++(0,1) and ++(0,-1) .. (2,2);
            \draw (2,0) .. controls ++(0,1) and ++(0,-1) .. (1,2);
            \node[draw,circle,fill,scale=0.2] (A) at (1.5,1) {$ $};
        \end{tikzpicture}
        } - (Q-1)\tikzinline{
        \begin{tikzpicture}[scale=0.6]
            \draw[dashed] (0,0) rectangle (3,2);
            \draw (1,0) -- (1,2);
            \draw (2,0) -- (2,2);
        \end{tikzpicture}
        }
    \end{align*}
\caption{The chromatic relations (1), (2), and (3), applied in that order, reduce a chromatic diagram to its simplest components in $\mathcal{C}_2$. Note that these equalities are equalities in the chromatic algebra $\mathcal{C}_2$, not in the free algebra.}
\label{fig:chromatic-relations}
\end{figure}

\subsection{Differences from Fendley and Krushkal}

Our definition of the chromatic algebra differs from the original definition presented by Fendley and Krushkal \cite[Definition~3.1]{fendley2010link}. In this section, we elaborate on these differences and explain how our definition fixes a small mistake in the pre-existing results.

\begin{proposition}
    The chromatic relations (1), (2), and (3) imply the following relation:
    \begin{enumerate}[start=4,label={(\arabic*)}]
        \item If $G$ contains a 1-valent vertex (in the interior of the rectangle) then $G = 0$.
    \end{enumerate}
\end{proposition}
\begin{proof}
    Let $G$ be a chromatic diagram with a 1-valent vertex $v$ connected to an edge $e$. By (3), we may construct a 2-valent vertex $w$ in the middle of edge $e$, splitting it into edges $e_{1}$ and $e_{2}$ such that $e_{1}$ connects $v$ and $w$. 
    
    Now, consider the chromatic diagram $G'$ obtained from $G$ by replacing $e$ with two edges $e_{1}$ and $e_{1}'$ between vertices $v$ and $w$. Applying (1) to $e_{1}'$ yields
    \begin{align*}
        G' = G'/e_{1}' - G' \setminus e_{1}' = G'/e_{1}' - G.
    \end{align*}
    By (3), the vertex $v$ in $G'$ can be removed, yielding a self-looping edge connected to vertex $w$. Thus, $G'$ is equivalent to $G'/e_{1}$, and $G = G'/e_{1}' - G' = 0$.
\end{proof}

We define the chromatic algebra slightly differently from Fendley and Krushkal \cite[Definition~3.1]{fendley2010link} in order to provide a more rigorous justification for the chromatic relations, especially concerning the addition and removal of 2-valent vertices. Now, we show why these changes were made and why the new definitions are still compatible with Fendley and Krushkal's other results.

According to Fendley and Krushkal, the chromatic algebra is defined by the following relations:
\begin{enumerate}[label={(\arabic*')},leftmargin=0.4in]
    \item If $e$ is an inner (not connected to boundary point) edge of a graph $G$ which is not a loop, then $G = G/e - G\setminus e$. 
    \item If $G$ contains an inner edge $e$ which is a loop, then $G = (Q-1) G \setminus e.$
    \item If $G$ contains a 1-valent vertex (in the interior of the rectangle) then $G = 0$.
\end{enumerate}

They claim that (1') and (3') imply that a 2-valent vertex may be deleted, and the two adjacent edges merged. Let this identity be (4'). The assumption that (1') and (3') imply (4') fails when a 2-valent vertex is adjacent to two outer edges, since neither edge can be contracted using (1'). In other words, there exists an edge case in which (1') and (3') do not imply (4'). An example is given in Figure~\ref{fig:edge-case}.

\begin{figure}[H]
    \begin{align*}
        \tikzinline
        {
        \begin{tikzpicture}[scale=0.5]
            \draw[dashed] (0,0) rectangle (2,2);
            \draw (1,0) -- (1,2);
            \node[draw,circle,fill,scale=0.2] (A) at (1,1) {$ $};
        \end{tikzpicture}
        }
        \neq
        \tikzinline
        {
        \begin{tikzpicture}[scale=0.5]
            \draw[dashed] (0,0) rectangle (2,2);
            \draw (1,0) -- (1,2);
        \end{tikzpicture}
        }
    \end{align*}
    \caption{Fendley and Krushkal's relation (4') \cite[Figure~5]{fendley2010link} cannot be applied to the left diagram.}
    \label{fig:edge-case}
\end{figure}
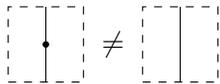


Our relations (1), (2), and (4) are equivalent to Fendley and Krushkal's relations (1'), (2'), and (3'), respectively. However, our relation (3) is equivalent to their relation (4'), which is not directly implied by their other relations (but appears to be assumed in their paper), so we assume our version of the definition to clear up any confusion.

\subsection{Chromatic Basis}

Amazingly, the relations $I_n$ reduce the infinite-dimensional free algebra $\mathcal{F}_n$ to the finite-dimensional algebra $\mathcal{C}_n$, with a finite basis that has been described in literature.

\begin{definition}
    Let $B_n \subset \mathcal{G}_n$ be the set of all $n$th-order chromatic diagrams without inner edges.
\end{definition}

\begin{proposition}[Fendley and Krushkal \protect{\cite[Lemma~4.1]{fendley2010link}}]
    The elements of $B_n$ form an additive basis of the chromatic algebra $\mathcal{C}_n$.
\end{proposition}

As an example, we list out $B_3$, a basis for $\mathcal{C}_3$, below.
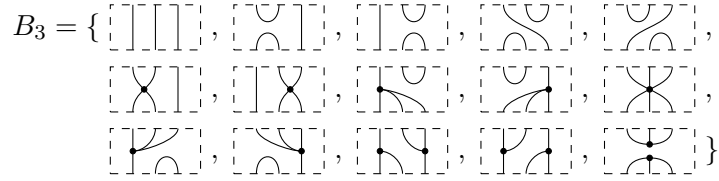
\begin{figure}[H]
    \begin{align*}
        B_3 = \{&
        \tikzinline{
        \begin{tikzpicture}[scale=0.3]
            \draw[dashed] (0,0) rectangle (4,2);
            \draw (1,0) -- (1,2);
            \draw (2,0) -- (2,2);
            \draw (3,0) -- (3,2);
        \end{tikzpicture}
        }, \tikzinline{
        \begin{tikzpicture}[scale=0.3]
            \draw[dashed] (0,0) rectangle (4,2);
            \draw (1,0) .. controls ++(0,1) and ++(0,1) .. (2,0);
            \draw (1,2) .. controls ++(0,-1) and ++(0,-1) .. (2,2);
            \draw (3,0) -- (3,2);
        \end{tikzpicture}
        }, \tikzinline{
        \begin{tikzpicture}[scale=0.3]
            \draw[dashed] (0,0) rectangle (4,2);
            \draw (1,0) -- (1,2);
            \draw (2,0) .. controls ++(0,1) and ++(0,1) .. (3,0);
            \draw (2,2) .. controls ++(0,-1) and ++(0,-1) .. (3,2);
        \end{tikzpicture}
        }, \tikzinline{
        \begin{tikzpicture}[scale=0.3]
            \draw[dashed] (0,0) rectangle (4,2);
            \draw (3,0) .. controls ++(0,1) and ++(0,-1) .. (1,2);
            \draw (1,0) .. controls ++(0,1) and ++(0,1) .. (2,0);
            \draw (2,2) .. controls ++(0,-1) and ++(0,-1) .. (3,2);
        \end{tikzpicture}
        }, \tikzinline{
        \begin{tikzpicture}[scale=0.3]
            \draw[dashed] (0,0) rectangle (4,2);
            \draw (1,0) .. controls ++(0,1) and ++(0,-1) .. (3,2);
            \draw (2,0) .. controls ++(0,1) and ++(0,1) .. (3,0);
            \draw (1,2) .. controls ++(0,-1) and ++(0,-1) .. (2,2);
        \end{tikzpicture}
        }, \\ &\tikzinline{
        \begin{tikzpicture}[scale=0.3]
            \draw[dashed] (0,0) rectangle (4,2);
            \draw (1,0) .. controls ++(0,1) and ++(0,-1) .. (2,2);
            \draw (2,0) .. controls ++(0,1) and ++(0,-1) .. (1,2);
            \draw (3,0) -- (3,2);
            \node[draw,circle,fill,scale=0.2] (A) at (1.5,1) {$ $};
        \end{tikzpicture}
        }, \tikzinline{
        \begin{tikzpicture}[scale=0.3]
            \draw[dashed] (0,0) rectangle (4,2);
            \draw (2,0) .. controls ++(0,1) and ++(0,-1) .. (3,2);
            \draw (3,0) .. controls ++(0,1) and ++(0,-1) .. (2,2);
            \draw (1,0) -- (1,2);
            \node[draw,circle,fill,scale=0.2] (A) at (2.5,1) {$ $};
        \end{tikzpicture}
        }, \tikzinline{
        \begin{tikzpicture}[scale=0.3]
            \draw[dashed] (0,0) rectangle (4,2);
            \draw (1,0) -- (1,2);
            \draw (2,2) .. controls ++(0,-1) and ++(0,-1) .. (3,2);
            \draw (2,0) .. controls ++(0,0.5) and ++(0.5,0) .. (1,1);
            \draw (3,0) .. controls ++(0,0.5) and ++(0.5,0) .. (1,1);
            \node[draw,circle,fill,scale=0.2] (A) at (1,1) {$ $};
        \end{tikzpicture}
        }, \tikzinline{
        \begin{tikzpicture}[scale=0.3]
            \draw[dashed] (0,0) rectangle (4,2);
            \draw (1,2) .. controls ++(0,-1) and ++(0,-1) .. (2,2);
            \draw (1,0) .. controls ++(0,0.5) and ++(-0.5,0) .. (3,1);
            \draw (2,0) .. controls ++(0,0.5) and ++(-0.5,0) .. (3,1);
            \draw (3,0) -- (3,2);
            \node[draw,circle,fill,scale=0.2] (A) at (3,1) {$ $};
        \end{tikzpicture}
        }, \tikzinline{
        \begin{tikzpicture}[scale=0.3]
            \draw[dashed] (0,0) rectangle (4,2);
            \draw (1,0) .. controls ++(0,1) and ++(0,-1) .. (3,2);
            \draw (3,0) .. controls ++(0,1) and ++(0,-1) .. (1,2);
            \draw (2,0) -- (2,2);
            \node[draw,circle,fill,scale=0.2] (A) at (2,1) {$ $};
        \end{tikzpicture}
        }, \\ & \tikzinline{
        \begin{tikzpicture}[scale=0.3]
            \draw[dashed] (0,0) rectangle (4,2);
            \draw (1,0) -- (1,2);
            \draw (2,0) .. controls ++(0,1) and ++(0,1) .. (3,0);
            \draw (2,2) .. controls ++(0,-0.5) and ++(0.5,0) .. (1,1);
            \draw (3,2) .. controls ++(0,-0.5) and ++(0.5,0) .. (1,1);
            \node[draw,circle,fill,scale=0.2] (A) at (1,1) {$ $};
        \end{tikzpicture}
        }, \tikzinline{
        \begin{tikzpicture}[scale=0.3]
            \draw[dashed] (0,0) rectangle (4,2);
            \draw (1,0) .. controls ++(0,1) and ++(0,1) .. (2,0);
            \draw (1,2) .. controls ++(0,-0.5) and ++(-0.5,0) .. (3,1);
            \draw (2,2) .. controls ++(0,-0.5) and ++(-0.5,0) .. (3,1);
            \draw (3,0) -- (3,2);
            \node[draw,circle,fill,scale=0.2] (A) at (3,1) {$ $};
        \end{tikzpicture}
        }, \tikzinline{
        \begin{tikzpicture}[scale=0.3]
            \draw[dashed] (0,0) rectangle (4,2);
            \draw (1,0) -- (1,2);
            \draw (2,0) .. controls ++(0,0.5) and ++(0.5,0) .. (1,1);
            \draw (2,2) .. controls ++(0,-0.5) and ++(-0.5,0) .. (3,1);
            \draw (3,0) -- (3,2);
            \node[draw,circle,fill,scale=0.2] (A) at (1,1) {$ $};
            \node[draw,circle,fill,scale=0.2] (B) at (3,1) {$ $};
        \end{tikzpicture}
        }, \tikzinline{
        \begin{tikzpicture}[scale=0.3]
            \draw[dashed] (0,0) rectangle (4,2);
            \draw (1,0) -- (1,2);
            \draw (2,0) .. controls ++(0,0.5) and ++(-0.5,0) .. (3,1);
            \draw (2,2) .. controls ++(0,-0.5) and ++(0.5,0) .. (1,1);
            \draw (3,0) -- (3,2);
            \node[draw,circle,fill,scale=0.2] (A) at (1,1) {$ $};
            \node[draw,circle,fill,scale=0.2] (B) at (3,1) {$ $};
        \end{tikzpicture}
        }, \tikzinline{
        \begin{tikzpicture}[scale=0.3]
            \draw[dashed] (0,0) rectangle (4,2);
            \draw (1,0) .. controls ++(0,0.5) and ++(-0.5,0) .. (2,0.7);
            \draw (1,2) .. controls ++(0,-0.5) and ++(-0.5,0) .. (2,1.3);
            \draw (2,0) -- (2,0.7);
            \draw (2,2) -- (2,1.3);
            \draw (3,0) .. controls ++(0,0.5) and ++(0.5,0) .. (2,0.7);
            \draw (3,2) .. controls ++(0,-0.5) and ++(0.5,0) .. (2,1.3);
            \node[draw,circle,fill,scale=0.2] (A) at (2,0.7) {$ $};
            \node[draw,circle,fill,scale=0.2] (A) at (2,1.3) {$ $};
        \end{tikzpicture}
        }
        \}
    \end{align*}
    \caption{The chromatic basis $B_3$.}
    \label{fig:b3-diagrams}
\end{figure}

\section{Dimension of the Chromatic Algebra}
\label{sec:dimension}

Although it is apparent that the basis $B_n$ consists of finitely many chromatic diagrams, the size of this set has not been computed before. 

In this section, we introduce a presentation of the chromatic algebra in terms of partitions. In particular, the diagrams in $B_n$ are in one-to-one correspondence with the noncrossing partitions of $2n$ points without singletons. This bijection was suggested in \cite[Section~4]{fendley2010link}, but a proof was not provided. Through this presentation, we determine that the dimension of the $n$th-order chromatic algebra $\mathcal{C}_{n}$ is the $2n$th Riordan number $R_{2n}$.

We now introduce the concepts of noncrossing partitions and singletons.
    A \emph{partition} of an $n$-element set is a division of $n$ elements into \emph{classes} (also called \emph{blocks}), where each element is assigned to exactly one class. Here the elements of a partition are vertices on a circle. For convenience, we number the vertices 1 through $n$ in counterclockwise order, starting from the bottom of the circumference, and assume they are evenly spaced.
    The \emph{graph of a partition}, or the \emph{circular representation} of the partition, is formed by cyclically joining the vertices of each block with edges, and is well-defined up to isotopy. A partition is \emph{noncrossing} if none of the resulting edges intersect. 
    A \emph{singleton} is a block of a partition containing only one element of the set.

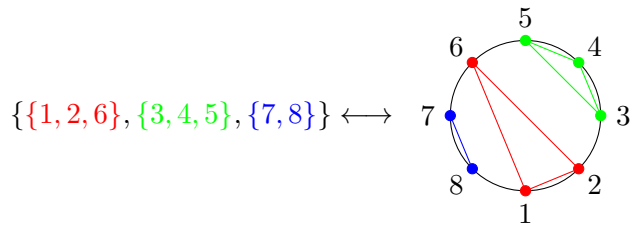
\begin{figure}[H]
    \begin{align*}
        \left\{ \textcolor{red}{\left\{1, 2, 6\right\}}, \textcolor{green}{\left\{ 3, 4, 5 \right\}}, \textcolor{blue}{\left\{ 7, 8\right\}} \right\} \longleftrightarrow
        \tikzinline{
        \begin{tikzpicture}
            \draw (0,0) circle[radius=1];
            \node[draw,circle,fill,scale=0.2,red] (A) at (0,-1) {1};
            \draw (0,-1.3) node {1};
            \node[draw,circle,fill,scale=0.2,red] (B) at (0.707,-0.707) {2};
            \draw (0.9191,-0.9191) node {2};
            \node[draw,circle,fill,scale=0.2,green] (C) at (1,0) {3};
            \draw (1.3,0) node {3};
            \node[draw,circle,fill,scale=0.2,green] (D) at (0.707,0.707) {4};
            \draw (0.9191,0.9191) node {4};
            \node[draw,circle,fill,scale=0.2,green] (E) at (0,1) {5};
            \draw (0,1.3) node {5};
            \node[draw,circle,fill,scale=0.2,red] (F) at (-0.707,0.707) {6};
            \draw (-0.9191,0.9191) node {6};
            \node[draw,circle,fill,scale=0.2,blue] (G) at (-1,0) {7};
            \draw (-1.3,0) node {7};
            \node[draw,circle,fill,scale=0.2,blue] (H) at (-0.707,-0.707) {8};
            \draw (-0.9191,-0.9191) node {8};
            \draw[red] (A) -- (B) -- (F) -- (A);
            \draw[green] (C) -- (D) -- (E) -- (C);
            \draw[blue] (G) -- (H);
        \end{tikzpicture}
        }
    \end{align*}
    \caption{A (noncrossing) partition of 8 points and its circular representation.}
    \label{fig:partition-circular-representation}
\end{figure}

Now, we claim the following property.

\begin{theorem*}[Theorem~\ref{thm:dimension}]
    The dimension of the chromatic algebra $\mathcal{C}_n$ is the $2n$th Riordan number $R_{2n}$.
\end{theorem*}

To prove this result, we make several conventions.

\begin{definition}
    We define a mapping $\alpha$ from $B_n$ to the set of partitions of $2n$ points as follows. Let $b$ be a diagram in $B_n$, and let $S$ be the set of its $2n$ boundary points. For each connected component in $b$, let the boundary points it connects form one block in the partition of $S$. Then, $\alpha(b)$ is the resulting partition of $S$.
\end{definition}

\begin{definition}
    For any positive integer $m$, let $P_{m}$ denote the denote the set of noncrossing partitions without singletons of $m$ points.
\end{definition}

\begin{lemma}
    The image of $\alpha$ lies within $P_{2n}$.
\end{lemma}
\begin{proof}
    Let $b$ be an element of $B_n$. Because each boundary point of $b$ must exist as part of a connected component with at least one other boundary point, $\alpha(b)$ has no singletons. Now, consider the circular representation of $\alpha(b)$. Because no two connected components of $b$ intersect, no two blocks in $\alpha(b)$ intersect, so the partition $\alpha(b)$ must be noncrossing. Because $\alpha(b)$ is necessarily noncrossing and without singletons, it is an element of $P_{2n}$.
\end{proof}

Henceforth, we treat $\alpha$ as a mapping from $B_n$ to $P_{2n}$. Next, we show that $\alpha$ is bijective, establishing the one-to-one correspondence between these two sets.

The \emph{adjacency matrix} of a labeled graph is a binary square matrix with rows and columns labeled by graph vertices, with position $(v_i,v_j)$ equal to the number of edges between vertices $v_i$ and $v_j$. We extend this notion to $n$th-order chromatic diagrams by introducing a well-defined labeling convention for their vertices. We label boundary points as vertices $1$ through $2n$, beginning from the bottom left and continuing in counterclockwise order. We then label the inner vertices with consecutively increasing integers, beginning with $2n+1$, in the order of the smallest-number boundary point they connect to. This labeling convention is illustrated in Figure~\ref{fig:vertex-labeling-order}.

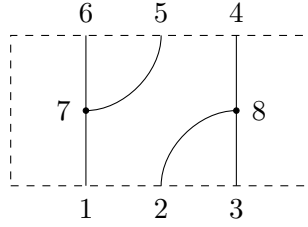
\begin{figure}[H]
    \begin{center}
        \begin{tikzpicture}
            \draw[dashed] (0,0) rectangle (4,2);
            \draw (1,0) -- (1,2);
            \draw (2,0) .. controls ++(0,0.5) and ++(-0.5,0) .. (3,1);
            \draw (2,2) .. controls ++(0,-0.5) and ++(0.5,0) .. (1,1);
            \draw (3,0) -- (3,2);
            \node[draw,circle,fill,scale=0.2] (A) at (1,1) {$ $};
            \node[draw,circle,fill,scale=0.2] (B) at (3,1) {$ $};
            \draw (1,-0.3) node {1};
            \draw (2,-0.3) node {2};
            \draw (3,-0.3) node {3};
            \draw (1,2.3) node {6};
            \draw (2,2.3) node {5};
            \draw (3,2.3) node {4};
            \draw (0.7,1) node {7};
            \draw (3.3,1) node {8};
        \end{tikzpicture}
    \end{center}
    \vspace{-20px}
    \caption{Labels of the boundary points and vertices of a third-order chromatic diagram.}
    \label{fig:vertex-labeling-order}
\end{figure}

\begin{lemma}
    \label{lem:adjacency-matrix-equivalence}
    Two chromatic diagrams in $B_n$ are equivalent if and only if their adjacency matrices are the same.
\end{lemma}
\begin{proof}
    First, note that self-looping edges are inner edges and thus not permitted.

    The only if direction is straightforward: because the vertices' order is fixed, two diagrams with different adjacency matrices must have different sets of edges. Therefore, they cannot be equivalent.

    We now prove the if direction. Let $G$ and $G'$ be two distinct chromatic diagrams. Note that because these diagrams belong to $B_n$, they cannot have inner edges. Moreover, because each boundary point can only connect to one edge, every outer edge must be adjacent to a segment of the rectangular boundary. Therefore, every face of the chromatic diagram that lies within the rectangle must contain one boundary edge in its perimeter. Now we consider, in counterclockwise order, the $2n$ segments of the boundary, beginning with the segment between points $1$ and $2$. For each boundary segment, initiate a counterclockwise walk that includes all edges in the face that connects to this boundary segment, beginning with itself. If $G$ and $G'$ are distinct, the sequences of edges in some walk must differ between the two. Therefore, there must exist a connection in $G$ which does not exist in $G'$, or vice versa. Then, the adjacency matrices of $G$ and $G'$ must differ.
\end{proof}

\begin{lemma}
    \label{lem:injective-partitions}
    The mapping $\alpha : B_n \to P_{2n}$ is injective.
\end{lemma}
\begin{proof}
    We prove by contradiction. Suppose that there exist two chromatic diagrams $a$ and $b$ that map to the same partition through $\alpha$. Because the number of inner vertices corresponds to the number of blocks in the partition with size at least 3, $a$ and $b$ must have the same number of inner vertices. Moreover, for each inner vertex, $a$ and $b$ must have the same subset of boundary points connected to that vertex by outer edges. Therefore, the two diagrams have the same set of outer edges.
    
    Because there are no inner edges in the chromatic diagrams in $B_n$, the number of inner vertices and the set of outer edges, taken together, completely determine the adjacency matrix of any chromatic diagram in $B_n$. Therefore, $a$ and $b$ have the same adjacency matrix, so by Lemma~\ref{lem:adjacency-matrix-equivalence}, they are equivalent.
\end{proof}

\begin{lemma}
    \label{lem:surjective-partitions}
    The mapping $\alpha : B_n \to P_{2n}$ is surjective.
\end{lemma}
\begin{proof}
    First, we establish a mapping $\beta$ from $P_{2n}$ to $B_n$. It is defined as follows: Let $p$ be a noncrossing partition without singletons of $2n$ points. Consider the circular representation of $p$. Define, for each block of this partition, a central vertex enclosed by the edges of the block. Draw edges from this central vertex to each point in its block such that each newly constructed edge lies within or on the block. Then, remove the edges cyclically joining the points in the block. Because no two blocks intersect, the new edges will not intersect either, so the resulting vertices and edges form a plane graph with $2n$ boundary points on the circle. Morph the boundary from a circle to the rectangle of a chromatic diagram, numbering the leftmost bottom boundary point ``1" and subsequently numbering in counterclockwise order. The mapping results in a chromatic diagram without inner edges, so $\beta(p)$ is necessarily an element of $B_n$.

    Each step in this mapping is well defined up to isotopy, and since elements of the codomain $B_n$ are only unique up to isotopy, this mapping is well defined.
    
    In $\beta(p)$ for $p \in P$, notice that the boundary points which belong to the same component are precisely the points which belong to the same block in the circular representation of $p$, so $\alpha(\beta(p))=p$ for every $p \in P$. Therefore, $\alpha$ is surjective.
\end{proof}

Now, we prove the one-to-one correspondence between diagrams in $B_n$ and noncrossing partitions of $2n$ points without singletons.
\begin{lemma}
    \label{lem:bijective-partitions}
    The mapping $\alpha : B_n \to P_{2n}$ is bijective.
\end{lemma}
\begin{proof}
    This result follows directly from Lemma~\ref{lem:injective-partitions} and Lemma~\ref{lem:surjective-partitions}.
\end{proof}

\begin{figure}[H]
    \begin{align*}
        \tikzinline{
        \begin{tikzpicture}[scale=0.8]
            \draw[dashed] (0,0) rectangle (5,2);
            \draw[red] (1,0) .. controls ++(0,0.5) and ++(-0.5,0) .. (2,1);
            \draw[red] (2,0) -- (2,1);
            \draw[red] (3,2) .. controls ++(0,-0.5) and ++(0.5,0) .. (2,1);
            \node[draw,circle,fill,scale=0.2,red] (A) at (2,1) {$ $};
            \draw[green] (3,0) .. controls ++(0,0.5) and ++(-0.5,0) .. (4,1);
            \draw[green] (4,0) -- (4,2);
            \node[draw,circle,fill,scale=0.2,green] (B) at (4,1) {$ $};
            \draw[blue] (1,2) .. controls ++(0,-1) and ++(0,-1) .. (2,2);
            \draw (1,-0.3) node {1};
            \draw (2,-0.3) node {2};
            \draw (3,-0.3) node {3};
            \draw (4,-0.3) node {4};
            \draw (1,2.3) node {8};
            \draw (2,2.3) node {7};
            \draw (3,2.3) node {6};
            \draw (4,2.3) node {5};
        \end{tikzpicture}
        }
        \longleftrightarrow
        \tikzinline{
        \begin{tikzpicture}
            \draw (0,0) circle[radius=1];
            \node[draw,circle,fill,scale=0.2,red] (A) at (0,-1) {1};
            \draw (0,-1.3) node {1};
            \node[draw,circle,fill,scale=0.2,red] (B) at (0.707,-0.707) {2};
            \draw (0.9191,-0.9191) node {2};
            \node[draw,circle,fill,scale=0.2,green] (C) at (1,0) {3};
            \draw (1.3,0) node {3};
            \node[draw,circle,fill,scale=0.2,green] (D) at (0.707,0.707) {4};
            \draw (0.9191,0.9191) node {4};
            \node[draw,circle,fill,scale=0.2,green] (E) at (0,1) {5};
            \draw (0,1.3) node {5};
            \node[draw,circle,fill,scale=0.2,red] (F) at (-0.707,0.707) {6};
            \draw (-0.9191,0.9191) node {6};
            \node[draw,circle,fill,scale=0.2,blue] (G) at (-1,0) {7};
            \draw (-1.3,0) node {7};
            \node[draw,circle,fill,scale=0.2,blue] (H) at (-0.707,-0.707) {8};
            \draw (-0.9191,-0.9191) node {8};
            \node[draw,circle,fill,scale=0.2,red] (AA) at (0,-0.3) {89};
            \node[draw,circle,fill,scale=0.2,green] (AB) at (0.4,0.4) {89};
            \draw[red] (A) -- (AA) -- (B);
            \draw[red] (F) -- (AA);
            \draw[green] (C) -- (AB) -- (D);
            \draw[green] (E) -- (AB);
            \draw[blue] (G) -- (H);
        \end{tikzpicture}
        }
        \longleftrightarrow
        \tikzinline{
        \begin{tikzpicture}
            \draw (0,0) circle[radius=1];
            \node[draw,circle,fill,scale=0.2,red] (A) at (0,-1) {1};
            \draw (0,-1.3) node {1};
            \node[draw,circle,fill,scale=0.2,red] (B) at (0.707,-0.707) {2};
            \draw (0.9191,-0.9191) node {2};
            \node[draw,circle,fill,scale=0.2,green] (C) at (1,0) {3};
            \draw (1.3,0) node {3};
            \node[draw,circle,fill,scale=0.2,green] (D) at (0.707,0.707) {4};
            \draw (0.9191,0.9191) node {4};
            \node[draw,circle,fill,scale=0.2,green] (E) at (0,1) {5};
            \draw (0,1.3) node {5};
            \node[draw,circle,fill,scale=0.2,red] (F) at (-0.707,0.707) {6};
            \draw (-0.9191,0.9191) node {6};
            \node[draw,circle,fill,scale=0.2,blue] (G) at (-1,0) {7};
            \draw (-1.3,0) node {7};
            \node[draw,circle,fill,scale=0.2,blue] (H) at (-0.707,-0.707) {8};
            \draw (-0.9191,-0.9191) node {8};
            \draw[red] (A) -- (B) -- (F) -- (A);
            \draw[green] (C) -- (D) -- (E) -- (C);
            \draw[blue] (G) -- (H);
        \end{tikzpicture}
        }
    \end{align*}
    \caption{Illustration of the bijection $\beta$. For clarity, this diagram includes a visual intermediate step which does not explicitly appear in the definition of $\beta$.}
    \label{fig:bijection-illustration}
\end{figure}
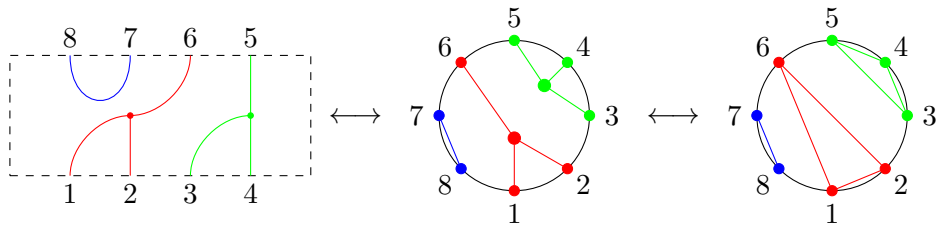

\begin{definition}[Riordan \protect{\cite[pp.~219]{riordan1975enumeration}}]
    The \emph{Riordan numbers} $R_{n}$ are an integer sequence given by $R_1=1$, $R_2=0$, and $(n+1) R_n = (n-1) (2R_{n-1} + 3R_{n-2})$ for all $n \geq 3$.
\end{definition}

\begin{remark}
    The $n$th Riordan number is given by the difference between the central coefficient and its predecessor in the trinomial expansion $(1+x+x^2)^n$, similar to how the $n$th Catalan number is the difference between the central coefficient and its predecessor in the binomial expansion $(1+x)^{2n}$ \cite[A005043]{oeis}. It turns out that these two combinatorial sequences are closely related.
\end{remark}

We now return to the proof of Theorem~\ref{thm:dimension}.
\begin{proof}[Proof of Theorem~\ref{thm:dimension}]
    By Lemma~\ref{lem:bijective-partitions}, the elements of $B_n$ are in one-to-one correspondence with the noncrossing partitions without singletons of $2n$ points. The size of this set is the Riordan number $R_{2n}$ \cite[Table~2, R3]{bernhart1999catalan}.
\end{proof}

The dimensions of small-order chromatic algebras are shown in Figure~\ref{fig:dimension}.

\begin{figure}[H]
    \begin{center}
        \begin{tabular}{|c||c|c|c|c|c|c|c|c|c|c|c|}
            \hline
            $n$ & 1 & 2 & 3 & 4 & 5 & 6 & 7 & 8 & 9 & 10 & 11
            \\ \hline\hline
            $|B_n|$ & 1 & 3 & 15 & 91 & 603 & 4213 & 30537 & 227475 & 1730787 & 13393689 & 105089229
            \\ \hline
        \end{tabular}
    \end{center}
    \caption{The dimensions of chromatic algebras of small order.}
    \label{fig:dimension}
\end{figure}

\section{Multiplicative Generating Set}
\label{sec:generating-set}

In the previous section, we discussed the chromatic basis, which is a additive generating set of the chromatic algebra. Results about the basis contribute useful information, including the dimension of the chromatic algebra and an efficient way to represent its elements.

In this section, we define a multiplicative generating set of this algebra, which contributes its own usefulness. In particular, once such a generating set is known, it becomes possible to redefine the algebra completely abstractly, in terms of abstract generator elements instead of chromatic diagrams. This would allow more studying of the chromatic algebra from an algebraic perspective, as opposed to the graph-theoretic perspective that existing literature, including our work, has used.

\begin{definition}
    Let $i,j,n$ be positive integers such that $1 \leq i < j \leq n$. Let $e_{i,j}^{n} \in B_n$ denote the diagram in which the all top and bottom boundary points between the $i$th from the left and the $j$th from the left, inclusive, are all connected to one inner vertex, and all other boundary points are paired with the corresponding boundary point on the opposite side by vertical edges. Let $E_n$ be the subset of $B_n$ containing all possible $e_{i,j}^{n}$.
\end{definition}
\begin{remark}
    Note that we implicitly allow the use of the identity in the generating set, following the same convention of the Temperley-Lieb algebra's generators \cite[(12.4.11)]{baxter2016exactly}.
\end{remark}

Figure~\ref{fig:generator} depicts an example of an element of $E_n$, and Figure~\ref{fig:generating-set} describes $E_1$, $E_2$, and $E_3$.

\begin{figure}[H]
    \begin{center}
        \begin{tikzpicture}
            \draw[dashed] (0,0) rectangle (8,2);
            \draw (1,0) -- (1,2);
            \draw (2,0) .. controls ++(0,1) and ++(0,-1) .. (5,2);
            \draw (3,0) .. controls ++(0,1) and ++(0,-1) .. (4,2);
            \draw (4,0) .. controls ++(0,1) and ++(0,-1) .. (3,2);
            \draw (5,0) .. controls ++(0,1) and ++(0,-1) .. (2,2);
            \draw (6,0) -- (6,2);
            \draw (7,0) -- (7,2);
            \node[draw,circle,fill,scale=0.2] (A) at (3.5,1) {$ $};
        \end{tikzpicture}
    \end{center}
    \caption{The element $e_{2,5}^7$ in $E_{7}$.}
    \label{fig:generator}
\end{figure}
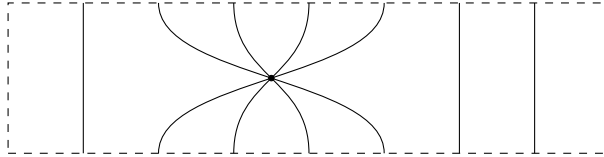

\begin{figure}[H]
    \begin{align*}
    E_1 &= \emptyset
    \\
    E_2 &= \left\{ 
    \tikzinline{
        \begin{tikzpicture}[scale=0.4]
            \draw[dashed] (0,0) rectangle (3,2);
            \draw (1,0) .. controls ++(0,1) and ++(0,-1) .. (2,2);
            \draw (2,0) .. controls ++(0,1) and ++(0,-1) .. (1,2);
            \node[draw,circle,fill,scale=0.2] (A) at (1.5,1) {$ $};
        \end{tikzpicture}
        }
    \right\}
        \\
        E_3 &= \left\{
        \tikzinline{
        \begin{tikzpicture}[scale=0.4]
            \draw[dashed] (0,0) rectangle (4,2);
            \draw (1,0) .. controls ++(0,1) and ++(0,-1) .. (2,2);
            \draw (2,0) .. controls ++(0,1) and ++(0,-1) .. (1,2);
            \draw (3,0) -- (3,2);
            \node[draw,circle,fill,scale=0.2] (A) at (1.5,1) {$ $};
        \end{tikzpicture}
        }, \tikzinline{
        \begin{tikzpicture}[scale=0.4]
            \draw[dashed] (0,0) rectangle (4,2);
            \draw (2,0) .. controls ++(0,1) and ++(0,-1) .. (3,2);
            \draw (3,0) .. controls ++(0,1) and ++(0,-1) .. (2,2);
            \draw (1,0) -- (1,2);
            \node[draw,circle,fill,scale=0.2] (A) at (2.5,1) {$ $};
        \end{tikzpicture}
        }, \tikzinline{
        \begin{tikzpicture}[scale=0.4]
            \draw[dashed] (0,0) rectangle (4,2);
            \draw (1,0) .. controls ++(0,1) and ++(0,-1) .. (3,2);
            \draw (3,0) .. controls ++(0,1) and ++(0,-1) .. (1,2);
            \draw (2,0) -- (2,2);
            \node[draw,circle,fill,scale=0.2] (A) at (2,1) {$ $};
        \end{tikzpicture}
        }
        \right\}
    \end{align*}
    \caption{The sets $E_1$, $E_2$, and $E_3$.}
    \label{fig:generating-set}
\end{figure}

Now, we claim the following property.

\begin{theorem*}[Theorem~\ref{thm:chromatic-generating-set}]
    The set $E_n$ generates $\mathcal{C}_n$.
\end{theorem*}
To prove this result, we make several conventions.


\begin{definition}
    Let $G$ be a diagram in $B_n$. Each component in $G$ with at least 3 boundary points contains exactly one (inner) vertex; let it be known as the \emph{vertex of the component}. For components with 2 boundary points (strands), add a 2-valent vertex at an arbitrary point on the strand and define this point to be the vertex of the component. (This is allowed because the chromatic relation implies that the class of $G$ in $\mathcal{C}_n$ is the same as the class of $G$ with a $2$-valent vertex added on the interior of some arbitrary edge.)
\end{definition}

\begin{definition}
    Let $G$ be a diagram in $\mathcal{C}_n$. A connected component in $G$ is called a \emph{top-isolated component} if it only contains top boundary points, a \emph{bottom-isolated component} if it only contains bottom boundary points, and a \emph{crossing component} if it contains both top and bottom boundary points.
\end{definition}

\begin{definition}
    Let $A$ and $B$ be two top-isolated or two bottom-isolated components. We say $A$ \emph{covers} $B$ if $A$, together with the top or bottom edge of the rectangle, circumscribes $B$.
\end{definition}

\begin{figure}[H]
    \begin{center}
        \begin{tikzpicture}
            \draw[dashed] (0,0) rectangle (6,2);

            \draw[red] (1,2) .. controls ++(0,-0.5) and ++(-0.5,0) .. (2,1.3);
            \draw[red] (2,2) -- (2,1.3);
            \draw[red] (3,2) .. controls ++(0,-0.5) and ++(0.5,0) .. (2,1.3);
            \node[red,draw,circle,fill,scale=0.2] (A) at (2,1.3) {$ $};

            \draw[green] (1,0) .. controls ++(0,1) and ++(0,1) .. (4,0);
            \draw[green,line width=5pt] (2,0) .. controls ++(0,0.5) and ++(0,0.5) .. (3,0);
            
            \draw[blue] (5,0) -- (5,2);
            \draw[blue] (4,2) .. controls ++(0,-0.5) and ++(-0.5,0) .. (5,1);
            \node[blue,draw,circle,fill,scale=0.2] (B) at (5,1) {$ $};

        \end{tikzpicture}
    \end{center}
    \caption{A chromatic diagram with top-isolated, bottom-isolated, and crossing components drawn in red, green, and blue, respectively. A covered bottom-isolated component is shown in bold.}
    \label{fig:component-types}
\end{figure}
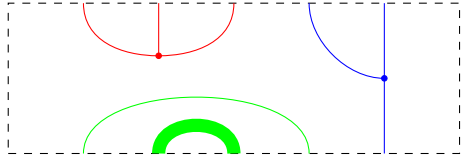

\begin{lemma}
    \label{lem:contraction-deletion-linear-combination}
    Let $G$ be a chromatic diagram, and let $j$ and $k$ be inner vertices in $G$ such that there is a multi-edge between vertices $j$ and $k$. Then, $G$ can be expressed as a linear combination of the diagram $G'$ obtained by contracting the multi-edge, and the diagram $G''$ obtained by deleting the multi-edge.
\end{lemma}
\begin{proof}
    This follows directly from repeated application of the chromatic relations (1) and (2) on each of the edges in the multi-edge.
\end{proof}

\begin{lemma}
    \label{lem:isolated-component-contraction}
Let $G$ be a diagram in $B_n$ with a top-isolated or bottom-isolated component which does not cover any other component. Then, $G$ can be expressed in terms of elements of $E_n$ and the diagram obtained from $G$ by contracting the vertex of this isolated component to the vertex of another component.
\end{lemma}
\begin{proof}
Let $L$ be the top-isolated or bottom-isolated component we consider. Without loss of generality, let $L$ be top-isolated.

Because $L$ does not cover any other components, its boundary points must all be consecutive. Let $i$ and $j$ be the positions, from the left, of its leftmost and rightmost boundary points.

Now, let $M$ be a component other than $L$ which is adjacent to $L$ through a face. We know that another component must exist because $L$ does not account for all of the boundary points in the diagram.

Let $G'$ be the diagram identical to $G$ except that the vertices of $L$ and $M$ have been contracted. Now, we concatenate the element $e_{i,j}^{n} \in E_n$ to the top of $G'$ to obtain a diagram $G_1$. Notice that the top half of $e_{i,j}^{n}$ is identical to $L$, so $G_1$ is identical to $G$ except that the vertices of $L$ and $M$ are connected by $j-i+1$ internal edges. By Lemma~\ref{lem:contraction-deletion-linear-combination}, therefore, $G_1$ is a linear combination of $G'$ and $G$. Therefore, $G$ can be expressed as a linear combination of $G'$ and $G_1 = e_{i,j}^{n} \times G'$, so it can be expressed in terms of elements of $E_n$ and $G'$.
\end{proof}

\begin{figure}[H]
    \begin{align*}
        G &= 
        \tikzinline{
            \begin{tikzpicture}[scale=0.6]
                \draw[dashed] (0,0) rectangle (6,2);
                \draw (1,2) .. controls ++(0,-0.5) and ++(-0.5,0) .. (2,1.5);
                \draw (2,2) -- (2,1.5);
                \draw (3,2) .. controls ++(0,-0.5) and ++(0.5,0) .. (2,1.5);
                \node[draw,circle,fill,scale=0.2] (A) at (2,1.5) {$ $};
                \draw (1,0) .. controls ++(0,1) and ++(0,1) .. (4,0);
                \draw (2,0) .. controls ++(0,0.5) and ++(0,0.5) .. (3,0);
                \draw (5,0) -- (5,2);
                \draw (4,2) .. controls ++(0,-0.5) and ++(-0.5,0) .. (5,1);
                \node[draw,circle,fill,scale=0.2] (B) at (5,1) {$ $};
            \end{tikzpicture}
        }, \hspace{4px} G' = \tikzinline{
            \begin{tikzpicture}[scale=0.6]
                \draw[dashed] (0,0) rectangle (6,2);
                \draw (1,2) .. controls ++(0,-0.5) and ++(-0.5,0) .. (5,1);
                \draw (2,2) .. controls ++(0,-0.5) and ++(-0.5,0) .. (5,1);
                \draw (3,2) .. controls ++(0,-0.5) and ++(-0.5,0) .. (5,1);
                \draw (1,0) .. controls ++(0,1) and ++(0,1) .. (4,0);
                \draw (2,0) .. controls ++(0,0.5) and ++(0,0.5) .. (3,0);
                \draw (5,0) -- (5,2);
                \draw (4,2) .. controls ++(0,-0.5) and ++(-0.5,0) .. (5,1);
                \node[draw,circle,fill,scale=0.2] (B) at (5,1) {$ $};
            \end{tikzpicture}
        }
        \\
        G_1 &= 
        \tikzinline{
            \begin{tikzpicture}[scale=0.6]
                \draw[dashed] (0,0) rectangle (6,2);
                \draw (1,2) .. controls ++(0,-0.5) and ++(-0.5,0) .. (2,1.5);
                \draw (2,2) -- (2,1.5);
                \draw (3,2) .. controls ++(0,-0.5) and ++(0.5,0) .. (2,1.5);
                \node[draw,circle,fill,scale=0.2] (A) at (2,1.5) {$ $};
                \draw (1,0) .. controls ++(0,1) and ++(0,1) .. (4,0);
                \draw (2,0) .. controls ++(0,0.5) and ++(0,0.5) .. (3,0);
                \draw (5,0) -- (5,2);
                \draw (4,2) .. controls ++(0,-0.5) and ++(0,0.5) .. (5,1);
                \draw (2,1.5) -- (5,1);
                \draw (2,1.5) .. controls ++(0,-0.5) and ++(-1,0) .. (5,1);
                \draw (2,1.5) .. controls ++(1,0) and ++(0,0.5) .. (5,1);
                \node[draw,circle,fill,scale=0.2] (B) at (5,1) {$ $};
            \end{tikzpicture}
        } \longrightarrow 
        \left\{  
        \tikzinline{
            \begin{tikzpicture}[scale=0.6]
                \draw[dashed] (0,0) rectangle (6,2);
                \draw (1,2) .. controls ++(0,-0.5) and ++(-0.5,0) .. (2,1.5);
                \draw (2,2) -- (2,1.5);
                \draw (3,2) .. controls ++(0,-0.5) and ++(0.5,0) .. (2,1.5);
                \node[draw,circle,fill,scale=0.2] (A) at (2,1.5) {$ $};
                \draw (1,0) .. controls ++(0,1) and ++(0,1) .. (4,0);
                \draw (2,0) .. controls ++(0,0.5) and ++(0,0.5) .. (3,0);
                \draw (5,0) -- (5,2);
                \draw (4,2) .. controls ++(0,-0.5) and ++(-0.5,0) .. (5,1);
                \node[draw,circle,fill,scale=0.2] (B) at (5,1) {$ $};
            \end{tikzpicture}
        },
        \tikzinline{
            \begin{tikzpicture}[scale=0.6]
                \draw[dashed] (0,0) rectangle (6,2);
                \draw (1,2) .. controls ++(0,-0.5) and ++(-0.5,0) .. (5,1);
                \draw (2,2) .. controls ++(0,-0.5) and ++(-0.5,0) .. (5,1);
                \draw (3,2) .. controls ++(0,-0.5) and ++(-0.5,0) .. (5,1);
                \draw (1,0) .. controls ++(0,1) and ++(0,1) .. (4,0);
                \draw (2,0) .. controls ++(0,0.5) and ++(0,0.5) .. (3,0);
                \draw (5,0) -- (5,2);
                \draw (4,2) .. controls ++(0,-0.5) and ++(-0.5,0) .. (5,1);
                \node[draw,circle,fill,scale=0.2] (B) at (5,1) {$ $};
            \end{tikzpicture}
        }
        \right\}
        \\
        G &\longrightarrow \{G' , G_1\}
        \\ G_1 &=
        \tikzinline{
            \begin{tikzpicture}[scale=0.6]
                \draw[dashed] (0,0) rectangle (6,2);
                \draw (1,2) .. controls ++(0,-0.5) and ++(-0.5,0) .. (2,1.5);
                \draw (2,2) -- (2,1.5);
                \draw (3,2) .. controls ++(0,-0.5) and ++(0.5,0) .. (2,1.5);
                \node[draw,circle,fill,scale=0.2] (A) at (2,1.5) {$ $};
                \draw (1,0) .. controls ++(0,1) and ++(0,1) .. (4,0);
                \draw (2,0) .. controls ++(0,0.5) and ++(0,0.5) .. (3,0);
                \draw (5,0) -- (5,2);
                \draw (4,2) .. controls ++(0,-0.5) and ++(0,0.5) .. (5,1);
                \draw (2,1.5) -- (5,1);
                \draw (2,1.5) .. controls ++(0,-0.5) and ++(-1,0) .. (5,1);
                \draw (2,1.5) .. controls ++(1,0) and ++(0,0.5) .. (5,1);
                \draw[dashed] (0,1) .. controls ++(5,0) and ++(-3,0) .. (6,1.6);
                \node[draw,circle,fill,scale=0.2] (B) at (5,1) {$ $};
            \end{tikzpicture}
        } =
        \tikzinline{
            \begin{tikzpicture}[scale=0.6]
                \draw[dashed] (0,0) rectangle (6,2);
                \draw (1,0) .. controls ++(0,1) and ++(0,-1) .. (3,2);
                \draw (3,0) .. controls ++(0,1) and ++(0,-1) .. (1,2);
                \draw (2,0) -- (2,2);
                \node[draw,circle,fill,scale=0.2] (A) at (2,1) {$ $};
                \draw (4,0) -- (4,2);
                \draw (5,0) -- (5,2);
            \end{tikzpicture}
        }
        \times
        \tikzinline{
            \begin{tikzpicture}[scale=0.6]
                \draw[dashed] (0,0) rectangle (6,2);
                \draw (1,2) .. controls ++(0,-0.5) and ++(-0.5,0) .. (5,1);
                \draw (2,2) .. controls ++(0,-0.5) and ++(-0.5,0) .. (5,1);
                \draw (3,2) .. controls ++(0,-0.5) and ++(-0.5,0) .. (5,1);
                \draw (1,0) .. controls ++(0,1) and ++(0,1) .. (4,0);
                \draw (2,0) .. controls ++(0,0.5) and ++(0,0.5) .. (3,0);
                \draw (5,0) -- (5,2);
                \draw (4,2) .. controls ++(0,-0.5) and ++(-0.5,0) .. (5,1);
                \node[draw,circle,fill,scale=0.2] (B) at (5,1) {$ $};
            \end{tikzpicture}
        } \\
        G &\longrightarrow \{G', e_{1,3}^5 \times G' \}
    \end{align*}
    \caption{The contraction of an isolated component, as performed in Lemma~\ref{lem:isolated-component-contraction}. A right arrow into a set of chromatic diagrams denotes that the chromatic diagram to the left of the arrow can be written as a linear combination of elements in the set.}
    \label{fig:isolated-component-contraction}
\end{figure}
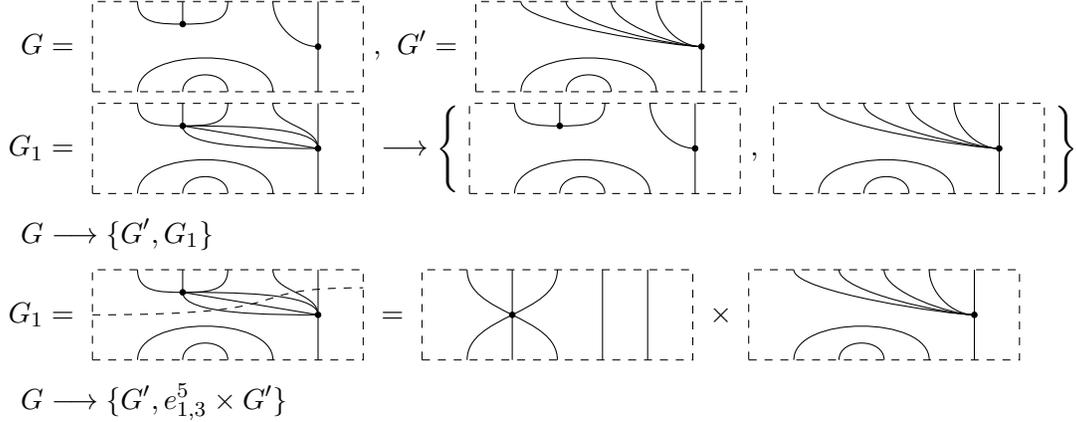

\begin{lemma}
\label{lem:isolated-component-removal}
Let $G$ be a diagram in $B_n$. Then, the diagram $G$ can be expressed in terms of elements of $E_n$ and a diagram where all components are crossing.
\end{lemma}
\begin{proof}
Note that $G$ may have top-isolated, bottom-isolated, and crossing components. Now, we prove by induction on the combined count of top-isolated and bottom-isolated components. The base case for this induction occurs when all the components of $G$ are crossing, and the combined count is zero.

For the inductive step, notice that we can always express $G$ in terms of a diagram with a lower combined count as follows:

Let $L$ be a top-isolated or bottom-isolated component which does not cover any other component. Without loss of generality, suppose $L$ is top-isolated. 


By Lemma~\ref{lem:isolated-component-contraction}, we may express $G$ in terms of $E_n$, and the diagram $G'$ obtained by contracting the vertex of $L$ and the vertex of an adjacent (through a face) component $M$. If $M$ is top-isolated or crossing, the number of top-isolated components decreases by one. If $M$ is bottom-isolated, the number of top-isolated components and the number of bottom-isolated components both decrease by one, reducing the total count by two. Therefore, $G'$ has a strictly lower combined count than $G$.

Therefore, the induction is complete.
\end{proof}

\begin{lemma}
\label{lem:only-crossing-reduction}
Let $G$ be a diagram in $B_n$ with only crossing components. Then, $G$ can be expressed in terms of elements of $E_n$ and diagrams in $B_n$ formed by taking some element of $B_i$ and $n-i$ vertical strands for $i<n$.
\end{lemma}
\begin{proof}
We induct on the number of connected components in the diagram with only crossing components.

The base case is that $G$ is the diagram in $B_n$ with only one connected component. This diagram is identical to $e_{1,n}^{n}$, so it can be directly generated by $E_n$.

Now, we consider the inductive step. Suppose that there $G$ has at least two connected components. Note that there is always a well-defined leftmost component, since all components are crossing. Denote this $L$, and let $a$ and $b$ be its number of top and bottom boundary points, respectively.

We first consider the case where $a=b$. Then, $L$ can be obtained from the element $e_{1,a}^n \in E_n$, and the rightmost section of $G$ excluding $L$ is identical to a diagram with only crossing components in $B_{n-a}$. In other words, $G$ can be written as the product of $e_{1,a}^n \in E_n$ and a diagram formed by $a$ vertical strands on the left, then a diagram in $B_{n-a}$ on the right.

\begin{figure}[H]
    \begin{align*}
        G = \tikzinline{
            \begin{tikzpicture}[scale=0.6]
                \draw[dashed] (0,0) rectangle (6,2);
                \draw (1,0) .. controls ++(0,1) and ++(0,-1) .. (3,2);
                \draw (3,0) .. controls ++(0,1) and ++(0,-1) .. (1,2);
                \draw (2,0) -- (2,2);
                \node[draw,circle,fill,scale=0.2] (A) at (2,1) {$ $};
                \draw (4,0) .. controls ++(0,1) and ++(0,-1) .. (5,2);
                \draw (5,0) .. controls ++(0,1) and ++(0,-1) .. (4,2);
                \node[draw,circle,fill,scale=0.2] (B) at (4.5,1) {$ $};
            \end{tikzpicture}
        } \longrightarrow
        \left\{
        \tikzinline{
            \begin{tikzpicture}[scale=0.6]
                \draw[dashed] (0,0) rectangle (4,2);
                \draw (1,0) .. controls ++(0,1) and ++(0,-1) .. (3,2);
                \draw (3,0) .. controls ++(0,1) and ++(0,-1) .. (1,2);
                \draw (2,0) -- (2,2);
                \node[draw,circle,fill,scale=0.2] (A) at (2,1) {$ $};
            \end{tikzpicture}
        },
        \tikzinline{
            \begin{tikzpicture}[scale=0.6]
                \draw[dashed] (0,0) rectangle (3,2);
                \draw (1,0) .. controls ++(0,1) and ++(0,-1) .. (2,2);
                \draw (2,0) .. controls ++(0,1) and ++(0,-1) .. (1,2);
                \node[draw,circle,fill,scale=0.2] (B) at (1.5,1) {$ $};
            \end{tikzpicture}
        }
        \right\}
    \end{align*}
\caption{The $a=b$ subcase in the inductive step of the proof of Lemma~\ref{lem:only-crossing-reduction}.}
\label{fig:only-crossing-reduction-abequal}
\end{figure}
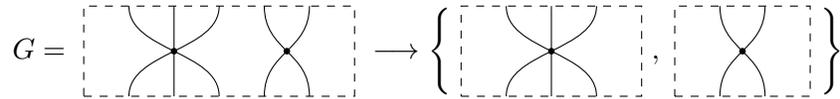

Now, without loss of generality, consider $a>b$. Denote the second leftmost component of $G$ as $M$, and let $G'$ be the diagram obtained from $G$ by contracting the vertices of $L$ and $M$. Let $G_1$ be the diagram obtained from $G$ by drawing $a-b$ edges between the vertices of $L$ and $M$. Then, by Lemma~\ref{lem:contraction-deletion-linear-combination}, $G_1$ can be expressed as a linear combination of $G$ and $G'$.

Let $G_2$ be the diagram obtained from $G$ by replacing $L$ with $b$ vertical strands, and connecting the leftover $a-b$ top boundary points to $M$. Then, $G_1$ can be obtained by concatenating the element $e_{1,a}^n \in E_n$ to the top of $G_2$. Therefore, $G_1$ can be expressed in terms of an element of $E_n$ and a diagram formed from vertical strands and $B_{n-b}$.

Therefore, $G$ can be expressed in terms of $G'$ and $G_1$, so it can be expressed in terms of an element of $E_n$, a diagram formed from vertical strands and $B_{n-b}$, and another diagram $G'$ with one less connected component than $G$. Thus, the induction is complete.
\end{proof}

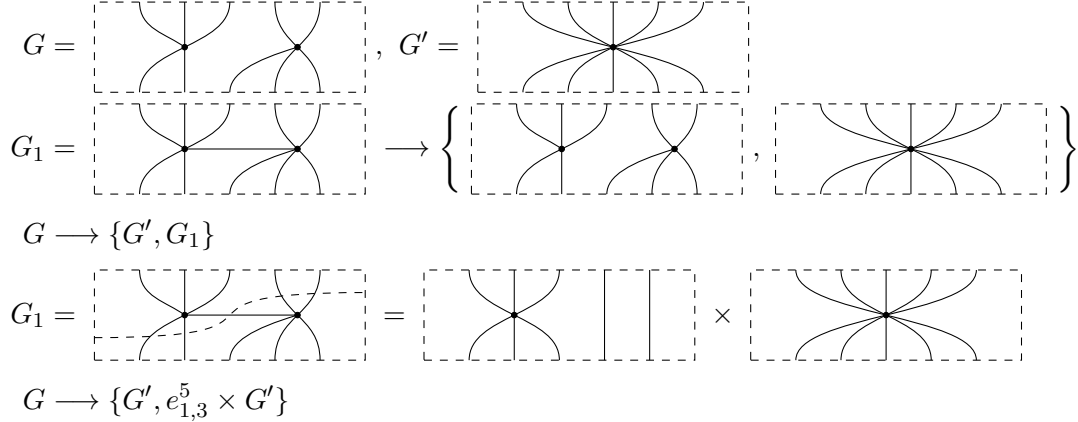
\begin{figure}[H]
    \begin{align*}
        G &= \tikzinline{
            \begin{tikzpicture}[scale=0.6]
                \draw[dashed] (0,0) rectangle (6,2);
                \draw (1,0) .. controls ++(0,1) and ++(0,-1) .. (3,2);
                \draw (1,2) .. controls ++(0,-0.6) and ++(-0.4,0.3) .. (2,1);
                \draw (2,0) -- (2,2);
                \node[draw,circle,fill,scale=0.2] (A) at (2,1) {$ $};
                \draw (3,0) .. controls ++(0,0.4) and ++(-1,-0.3) .. (4.5,1);
                \draw (4,0) .. controls ++(0,1) and ++(0,-1) .. (5,2);
                \draw (5,0) .. controls ++(0,1) and ++(0,-1) .. (4,2);
                \node[draw,circle,fill,scale=0.2] (B) at (4.5,1) {$ $};
            \end{tikzpicture}
        }, \hspace{4px} G' = \tikzinline{
            \begin{tikzpicture}[scale=0.6]
                \draw[dashed] (0,0) rectangle (6,2);
                \draw (1,0) .. controls ++(0,1) and ++(0,-1) .. (5,2);
                \draw (2,0) .. controls ++(0,1) and ++(0,-1) .. (4,2);
                \draw (3,0) -- (3,2);
                \draw (4,0) .. controls ++(0,1) and ++(0,-1) .. (2,2);
                \draw (5,0) .. controls ++(0,1) and ++(0,-1) .. (1,2);
                \node[draw,circle,fill,scale=0.2] (B) at (3,1) {$ $};
            \end{tikzpicture} 
        }
        \\
        G_1 &= \tikzinline{
            \begin{tikzpicture}[scale=0.6]
                \draw[dashed] (0,0) rectangle (6,2);
                \draw (1,0) .. controls ++(0,1) and ++(0,-1) .. (3,2);
                \draw (1,2) .. controls ++(0,-0.6) and ++(-0.4,0.3) .. (2,1);
                \draw (2,0) -- (2,2);
                \node[draw,circle,fill,scale=0.2] (A) at (2,1) {$ $};
                \draw (3,0) .. controls ++(0,0.4) and ++(-1,-0.3) .. (4.5,1);
                \draw (4,0) .. controls ++(0,1) and ++(0,-1) .. (5,2);
                \draw (5,0) .. controls ++(0,1) and ++(0,-1) .. (4,2);
                \node[draw,circle,fill,scale=0.2] (B) at (4.5,1) {$ $};
                \draw (2,1) -- (4.5,1);
            \end{tikzpicture}
        }
        \longrightarrow \left\{ 
        \tikzinline{
            \begin{tikzpicture}[scale=0.6]
                \draw[dashed] (0,0) rectangle (6,2);
                \draw (1,0) .. controls ++(0,1) and ++(0,-1) .. (3,2);
                \draw (1,2) .. controls ++(0,-0.6) and ++(-0.4,0.3) .. (2,1);
                \draw (2,0) -- (2,2);
                \node[draw,circle,fill,scale=0.2] (A) at (2,1) {$ $};
                \draw (3,0) .. controls ++(0,0.4) and ++(-1,-0.3) .. (4.5,1);
                \draw (4,0) .. controls ++(0,1) and ++(0,-1) .. (5,2);
                \draw (5,0) .. controls ++(0,1) and ++(0,-1) .. (4,2);
                \node[draw,circle,fill,scale=0.2] (B) at (4.5,1) {$ $};
            \end{tikzpicture}
        }, \tikzinline{
            \begin{tikzpicture}[scale=0.6]
                \draw[dashed] (0,0) rectangle (6,2);
                \draw (1,0) .. controls ++(0,1) and ++(0,-1) .. (5,2);
                \draw (2,0) .. controls ++(0,1) and ++(0,-1) .. (4,2);
                \draw (3,0) -- (3,2);
                \draw (4,0) .. controls ++(0,1) and ++(0,-1) .. (2,2);
                \draw (5,0) .. controls ++(0,1) and ++(0,-1) .. (1,2);
                \node[draw,circle,fill,scale=0.2] (B) at (3,1) {$ $};
            \end{tikzpicture} 
        }
        \right\}
        \\
        G &\longrightarrow \{ G', G_1\}
        \\ G_1 &= \tikzinline{
            \begin{tikzpicture}[scale=0.6]
                \draw[dashed] (0,0) rectangle (6,2);
                \draw (1,0) .. controls ++(0,1) and ++(0,-1) .. (3,2);
                \draw (1,2) .. controls ++(0,-0.6) and ++(-0.4,0.3) .. (2,1);
                \draw (2,0) -- (2,2);
                \node[draw,circle,fill,scale=0.2] (A) at (2,1) {$ $};
                \draw (3,0) .. controls ++(0,0.4) and ++(-1,-0.3) .. (4.5,1);
                \draw (4,0) .. controls ++(0,1) and ++(0,-1) .. (5,2);
                \draw (5,0) .. controls ++(0,1) and ++(0,-1) .. (4,2);
                \node[draw,circle,fill,scale=0.2] (B) at (4.5,1) {$ $};
                \draw (2,1) -- (4.5,1);
                \draw[dashed] (0,0.5) .. controls ++(5,0) and ++(-5,0) .. (6,1.5);
            \end{tikzpicture}
        }
        = \tikzinline{
            \begin{tikzpicture}[scale=0.6]
                \draw[dashed] (0,0) rectangle (6,2);
                \draw (1,0) .. controls ++(0,1) and ++(0,-1) .. (3,2);
                \draw (3,0) .. controls ++(0,1) and ++(0,-1) .. (1,2);
                \draw (2,0) -- (2,2);
                \node[draw,circle,fill,scale=0.2] (A) at (2,1) {$ $};
                \draw (4,0) -- (4,2);
                \draw (5,0) -- (5,2);
            \end{tikzpicture}
        }
        \times
        \tikzinline{
            \begin{tikzpicture}[scale=0.6]
                \draw[dashed] (0,0) rectangle (6,2);
                \draw (1,0) .. controls ++(0,1) and ++(0,-1) .. (5,2);
                \draw (2,0) .. controls ++(0,1) and ++(0,-1) .. (4,2);
                \draw (3,0) -- (3,2);
                \draw (4,0) .. controls ++(0,1) and ++(0,-1) .. (2,2);
                \draw (5,0) .. controls ++(0,1) and ++(0,-1) .. (1,2);
                \node[draw,circle,fill,scale=0.2] (B) at (3,1) {$ $};
            \end{tikzpicture} 
        } \\
        G &\longrightarrow \{G', e_{1,3}^5 \times G' \}
    \end{align*}
\caption{The $a>b$ subcase in the inductive step of the proof of Lemma~\ref{lem:only-crossing-reduction}. A right arrow into a set of chromatic diagrams denotes that the chromatic diagram to the left of the arrow can be written as a linear combination of elements in the set.}
\label{fig:only-crossing-reduction-general}
\end{figure}

\begin{proof}[Proof of Theorem~\ref{thm:chromatic-generating-set}]
    Since $B_n$ is an additive basis, every element of $\mathcal{C}_n$ can be expressed as a linear combination of elements of $B_n$. Now, we prove by inducting on $n$ that $E_n$ generates every element of $B_n$.

    The base case is $B_1$, which shares its only element with $E_1$.

    For the inductive step, suppose that the result has been proven for all $B_m$ where $m<n$. Note that by Lemma~\ref{lem:isolated-component-removal} and Lemma~\ref{lem:only-crossing-reduction}, every diagram in $B_n$ can be expressed in terms of elements of $E_n$ and diagrams formed from vertical strands and elements of $B_i$, for $i<n$. All of these such diagrams can be expressed in terms of elements of $E_n$.
\end{proof}

\begin{remark}
    Note that the generating set $E_n$ has size $\binom{n}{2}$. This is very interesting to use, since the basis $B_n$ grows exponentially, while the generating set $E_n$ grows quadratically. The sizes of $E_n$ for small-order chromatic algebras are shown in Figure~\ref{fig:number-of-generators}. 
\end{remark}

\begin{figure}[H]
    \begin{center}
        \begin{tabular}{|c||c|c|c|c|c|c|c|c|c|c|c|}
            \hline
            $n$ & 1 & 2 & 3 & 4 & 5 & 6 & 7 & 8 & 9 & 10 & 11
            \\ \hline\hline
            $|E_n|$ & 0 & 1 & 3 & 6 & 10 & 15 & 21 & 28 & 36 & 45 & 55
            \\ \hline
        \end{tabular}
    \end{center}
    \caption{The sizes of $|E_n|$ for chromatic algebras of small order.}
    \label{fig:number-of-generators}
\end{figure}
Note that we do \textit{not} claim that $E_n$ is minimal as a generating set! This remains to be proven.

\begin{conjecture}
   The set $E_n$ is a minimal generating set in the weak sense: no proper subset of $E_n$ generates $\mathcal{C}_n$.
\end{conjecture}
\begin{conjecture}
   The set $E_n$ is a minimal generating set in the strong sense: $E_n$ has the lowest cardinality of any generating set of $\mathcal{C}_n$.
\end{conjecture}

\begin{remark}
    Prof.~Pavel Etingof raised another interesting property to study: the degree of the generating set $E_n$, defined as follows. We know that $\mathcal{C}_n$ is a quotient of the (noncommutative) ring $\mathbb{C}((Q))\langle E_n\rangle$, say by a kernel $I$. An interesting question is to find a generating set for $I$, and determine the minimum over all generating sets for $I$ of the maximal degree of its polynomials (in $e_{i,j}^{n}$). This assigns a ``complexity" to the generating set.
\end{remark}

\section*{Acknowledgements}

I would like to thank my PRIMES mentor, Merrick Cai, for his crucial and continued guidance throughout this research.

I would also like to thank Dr.~Minh-Tâm Trinh for his expert guidance and ideas, including introducing this area of study to Merrick and me.

Finally, I would like to thank the organizers of the PRIMES program, especially Prof.~Pavel Etingof, Dr.~Slava Gerovitch, Dr.~Felix Gotti, and Dr.~Tanya Khovanova, for giving me and many other students the opportunity to conduct guided math research.



\bibliographystyle{plain}

\end{document}